\documentclass[12pt,a4paper,intlimits,sumlimits]{amsart}

\usepackage{amsmath, amsthm, mathtools}
\usepackage{amsfonts}
\usepackage[initials, nobysame]{amsrefs}
\usepackage{graphicx}
\usepackage{framed}
\usepackage{amssymb}
\usepackage{esvect}
\usepackage{xcolor}
\usepackage{eucal}
\usepackage{mathrsfs}
\usepackage{hyperref}
\usepackage[english]{babel}
\usepackage{mathrsfs}
\usepackage{esint}

\def \R {\mathbb{R}}

\theoremstyle{definition}
\newtheorem{definition}{Definition}[section]

\theoremstyle{plain}
\newtheorem{theorem}[definition]{Theorem}
\newtheorem{proposition}[definition]{Proposition}
\newtheorem{lemma}[definition]{Lemma}
\newtheorem{corollary}[definition]{Corollary}

\numberwithin{equation}{section}

\usepackage{geometry}
\renewcommand{\epsilon}{\varepsilon}

\renewcommand{\leq}{\leqslant}
\renewcommand{\le}{\leqslant}
\renewcommand{\geq}{\geqslant}
\renewcommand{\ge}{\geqslant}

\allowdisplaybreaks

\title[Optimal embedding results for fractional Sobolev spaces]{Optimal embedding results \\for fractional Sobolev spaces}

\author[S. Dipierro, E. Proietti Lippi, C. Sportelli and E. Valdinoci]{Serena Dipierro, Edoardo Proietti Lippi, Caterina Sportelli and Enrico Valdinoci}

\address{Department of Mathematics and Statistics
\newline\indent University of Western Australia \newline\indent
35 Stirling Highway, WA 6009 Crawley, Australia.\newline
\newline\indent
\tt serena.dipierro@uwa.edu.au \newline\indent
\tt edoardo.proiettilippi@uwa.edu.au \newline\indent
\tt caterina.sportelli@uwa.edu.au \newline\indent
\tt enrico.valdinoci@uwa.edu.au}

\begin{document}

\maketitle

\begin{abstract}
This paper deals with the fractional Sobolev spaces~$W^{s, p}(\Omega)$, with~$s\in (0, 1]$ and~$p\in[1,+\infty]$. Here, we use the interpolation results in~\cite{MR3813967} to provide suitable conditions on the exponents~$s$ and~$p$ so that the spaces~$W^{s, p}(\Omega)$ realize a continuous embedding when either~$\Omega=\R^N$ or~$\Omega$ is any open and bounded domain with Lipschitz boundary.

Our results enhance the classical continuous embedding and, when~$\Omega$ is any open bounded domain with Lipschitz boundary, we also improve the classical compact embeddings.

All the results stated here are proved to be optimal. Also, our strategy does not require the use of Besov or other interpolation spaces. 
\end{abstract}

\tableofcontents

\section{Introduction and statement of the main results}

In this paper we prove several results concerning the continuous and compact embeddings of fractional Sobolev spaces~$W^{s, p}(\Omega)$, where~$s\in (0, 1]$, $p\in [1, +\infty]$ and 
\begin{equation}\label{Omega}
{\mbox{either~$\Omega=\R^N$ or~$\Omega\subset\R^N$ is an open and bounded domain with Lipschitz boundary.}}
\end{equation}
All of the embeddings provided here are also shown to be optimal.
In particular, our work completes the classification of the fractional Sobolev embeddings by effectively proposing (depending on the interaction between the exponents~$s$, $p$ and~$N$) a series of optimal embedding results.

We remark that all the proofs that we present here are self-contained. Moreover, we do not make use of Besov spaces,
but rather we use the interpolation results provided by Brezis and Mironescu in~\cite{MR3813967}.

Besides, we point out that the present work refines some of the results on continuous embeddings obtained in~\cite{MR3910033} and~\cite{MR3990737}. For this, we provide some specific results of compact embeddings for the spaces~$W^{s, p}(\Omega)$ which go beyond the classical ones e.g. in~\cite{MR2944369} and~\cite[Theorem~16.1]{MR350177}.

The embedding results that
we present here can be of particular interest in the study of problems involving fractional~$p$-Laplacian operators. Indeed, they may enrich the literature and open up new scenarios.
To make an example, one could be interested in reconsidering the results in~\cite{TUTTI3} in light of the embeddings presented here and possibly choose some less restrictive assumptions.
\medskip

To establish our main results, we recall~\eqref{Omega} and we introduce the following notation. For~$s\in [0, 1]$ and~$p\in [1, +\infty]$, we define 
\[
[u]_{W^{s,p}(\Omega)}:=
\begin{cases}
\|u\|_{L^p(\Omega)}  &\mbox{ if } s=0,
\\ \\
\displaystyle\left(c_{N,s,p}\iint_{\Omega\times\Omega}\frac{|u(x)-u(y)|^p}{|x-y|^{N+sp}}\,dx\,dy \right)^{1/p} &\mbox{ if } s\in(0,1),
\\ \\
\|\nabla u\|_{L^p(\Omega)}  &\mbox{ if } s=1.
\end{cases}
\]
Here, $c_{N,s,p}$ denotes a normalizing constant, given explicitly by
\begin{equation*}
c_{N,s,p}:=\frac{s\,2^{2s-1}\,\Gamma\left(\frac{ps+p+N-2}{2}\right)}{\pi^{N/2}\,\Gamma(1-s)},
\end{equation*}
see e.g.~\cite[page~130]{MR3473114} and the references therein.

The normalizing constant is determined in such a way that
\[
\lim_{s\searrow0}[u]_{W^{s,p}(\Omega)}=[u]_{W^{0,p}(\Omega)}=\|u\|_{L^p(\Omega)}
\qquad{\mbox{and}}\qquad\lim_{s\nearrow1}[u]_{W^{s,p}(\Omega)}=[u]_{W^{1,p}(\Omega)}=\|\nabla u\|_{L^p(\Omega)},
\]
see~\cite{MR3586796}.

As customary, we define the fractional Sobolev space
\[
W^{s, p}(\Omega):=\left\{u\in L^p(\Omega) \mbox{ such that } [u]_{W^{s,p}(\Omega)}<+\infty \right\}
\]
endowed with the norm
\begin{equation}\label{normaWsp}
\|u\|_{W^{s, p}(\Omega)}:=
\begin{cases}
\|u\|_{L^p(\Omega)}  &\mbox{ if } s=0,
\\
\left([u]^p_{W^{s, p}(\Omega)} + \|u\|^p_{L^p(\Omega)}\right)^{\frac1p} &\mbox{ if } s\in (0, 1),
\\
\left(\|\nabla u\|^p_{L^p(\Omega)} + \|u\|^p_{L^p(\Omega)}\right)^{\frac1p} &\mbox{ if } s=1.
\end{cases}
\end{equation}

We present our main results distinguishing three different cases.

\subsection{The case $sp<N$}
%%In this case, we can define the ``Sobolev exponent'' as
%%$$ p^*_s:=\frac{Np}{N-sp}.$$
%%
In this setting, when~$\Omega=\R^N$,
our main result reads as follows:

\begin{theorem}\label{teorema1}
Let~$s\in [0,1]$ and~$p\in [1, +\infty)$ be such that~$sp<N$ and~$s\ne p$. Let~$\widetilde{s}$ and~$\widetilde{p}$ satisfy
\begin{equation}\label{insiemeimmersioneRN}
\begin{cases}
0\le \widetilde{s}\le s,\\
p \le\widetilde{p}\le 
\frac{Np}{N-(s-\widetilde s)p}.
%% \dfrac{spp^*_s}{sp+(p^*_s-p)\widetilde{s}}.
\end{cases}
\end{equation}

Then, there exists a positive constant~$C=C(N,s,p,\widetilde{s},\widetilde{p})$ such that, for any~$u\in W^{s,p}(\R^N)$,
\begin{equation}\label{inequality1}
\|u\|_{W^{\widetilde{s},\widetilde{p}}(\R^N)}
\le C\|u\|_{W^{s,p}(\R^N)},
\end{equation}
namely, the space~$W^{s,p}(\R^N)$ is continuously embedded in~$W^{\widetilde{s},\widetilde{p}}(\R^N)$.

Moreover, the curve~$\gamma:[0,1]\to \R^2$ defined as
\begin{equation}\label{curva1}
\gamma(\theta):=(s_\theta,p_\theta)=
\left(\theta \widetilde{s}+(1-\theta)s,
\frac{p\widetilde{p}}{\widetilde{p}+\theta(p-\widetilde{p})}  \right)
\end{equation}
is such that, if~$0\le \theta_1 \le \theta_2 \le 1$, 
then the space~$W^{s_{\theta_1},p_{\theta_1}}(\R^N)$
is continuously embedded\footnote{In the spirit of~\cite[formula~(1.5)]{MR3813967}, we point out that the same result holds replacing~$\R^N$ with a halfspace.} in~$W^{s_{\theta_2},p_{\theta_2}}(\R^N)$.
\end{theorem}

The set of points satisfying~\eqref{insiemeimmersioneRN} and realizing the embedding stated in Theorem~\ref{teorema1} is illustrated in Figure~\ref{red}.

\begin{figure}[h]
\begin{center}
\includegraphics[scale=.31]{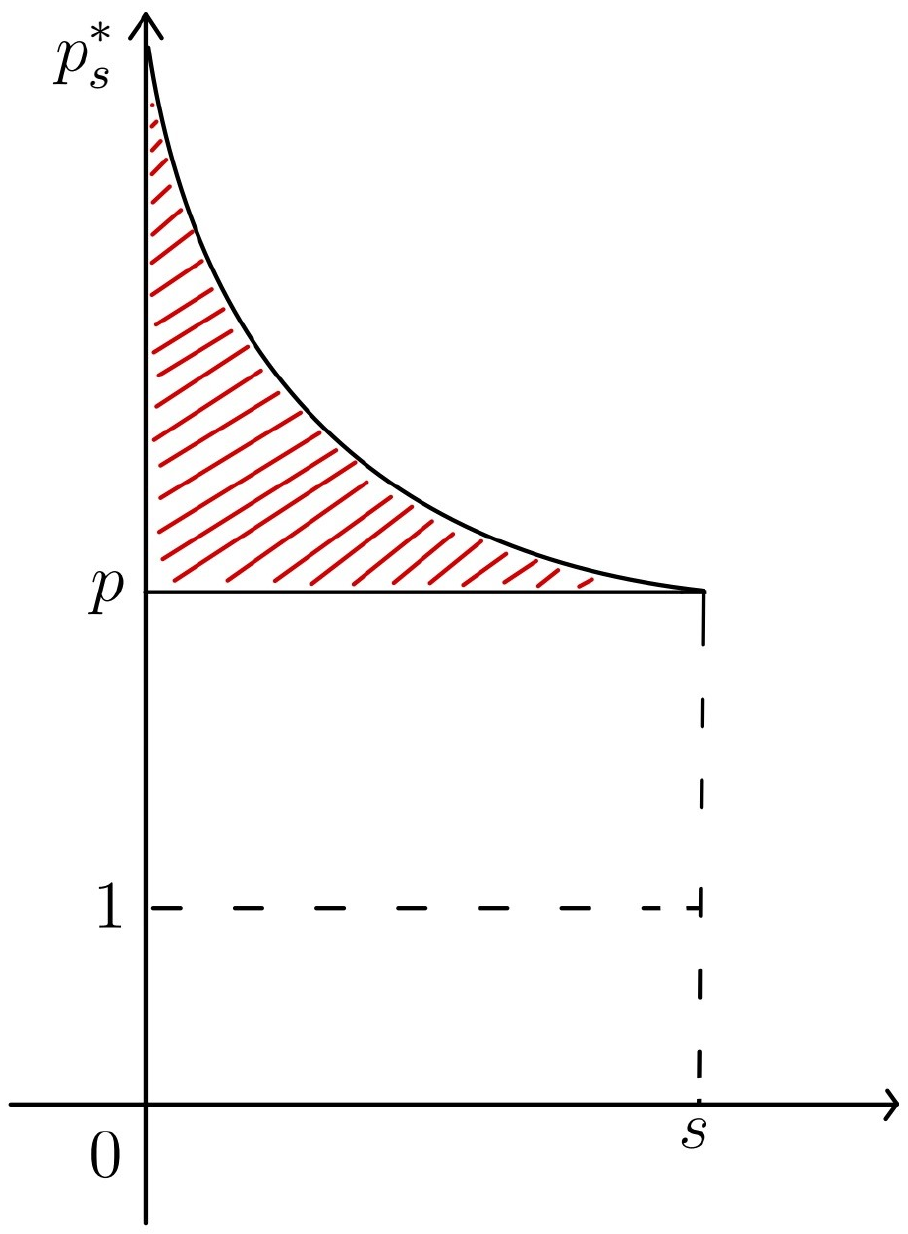}
\end{center}
\caption{The set of points satisfying~\eqref{insiemeimmersioneRN}.}
\label{red}
\end{figure}

When~$\Omega$ is an open and bounded domain with Lipschitz boundary, we have the following:

\begin{theorem}\label{teorema2}
Let~$s\in (0,1]$ and~$p\in [1, +\infty)$ be such that~$sp<N$ and~$s\ne p$.
Let~$\Omega\subset \R^N$ be an open and bounded domain with Lipschitz boundary.
Let~$\widetilde{s}$ and~$\widetilde{p}$ satisfy\footnote{We point out that, if~$\widetilde s = s$ and~$\widetilde p = p$, then the result is trivial.  Moreover, if~$\widetilde s = s=1$, then the desired result holds true for~$1\le \widetilde p\le p$. \label{footnotetrivial}}
\begin{equation}\label{insiemeimmersioneOmega}
\begin{cases}
0\le \widetilde{s}< s,\\
1 \le \widetilde{p}\le
\frac{Np}{N-(s-\widetilde s)p}.
%%\le \dfrac{spp^*_s}{sp+(p^*_s-p)\widetilde{s}}.
\end{cases}
\end{equation}

Then, there exists a positive constant~$C=C(N,s,p,\widetilde{s},\widetilde{p},\Omega)$ such that, for any~$u\in W^{s,p}(\Omega)$,
\begin{equation}\label{inequality2}
\|u\|_{W^{\widetilde{s},\widetilde{p}}(\Omega)}\le C\|u\|_{W^{s,p}(\Omega)},
\end{equation}
namely, the space~$W^{s,p}(\Omega)$ is continuously embedded in~$W^{\widetilde{s},\widetilde{p}}(\Omega)$.

Moreover, the curve~$\gamma$ defined in~\eqref{curva1} is such that, if~$0\le \theta_1 \le \theta_2 \le 1$, 
then the space $W^{s_{\theta_1},p_{\theta_1}}(\Omega)$
is continuously embedded in~$W^{s_{\theta_2},p_{\theta_2}}(\Omega)$.
\end{theorem}

We point out that Theorem~\ref{teorema2}
has been already proved
in the case~$\widetilde p = p=2$ in~\cite[Lemma~2.1]{MR4736013},
and in the case~$\widetilde p = p$ for any~$p\in(1,N)$  in~\cite[Theorem~3.2]{DPSV2}.

The set of points satisfying~\eqref{insiemeimmersioneOmega} and realizing the embedding stated in Theorem~\ref{teorema2} is visualized in Figure~\ref{green}.

\begin{figure}[h]
\begin{center}
\includegraphics[scale=.28]{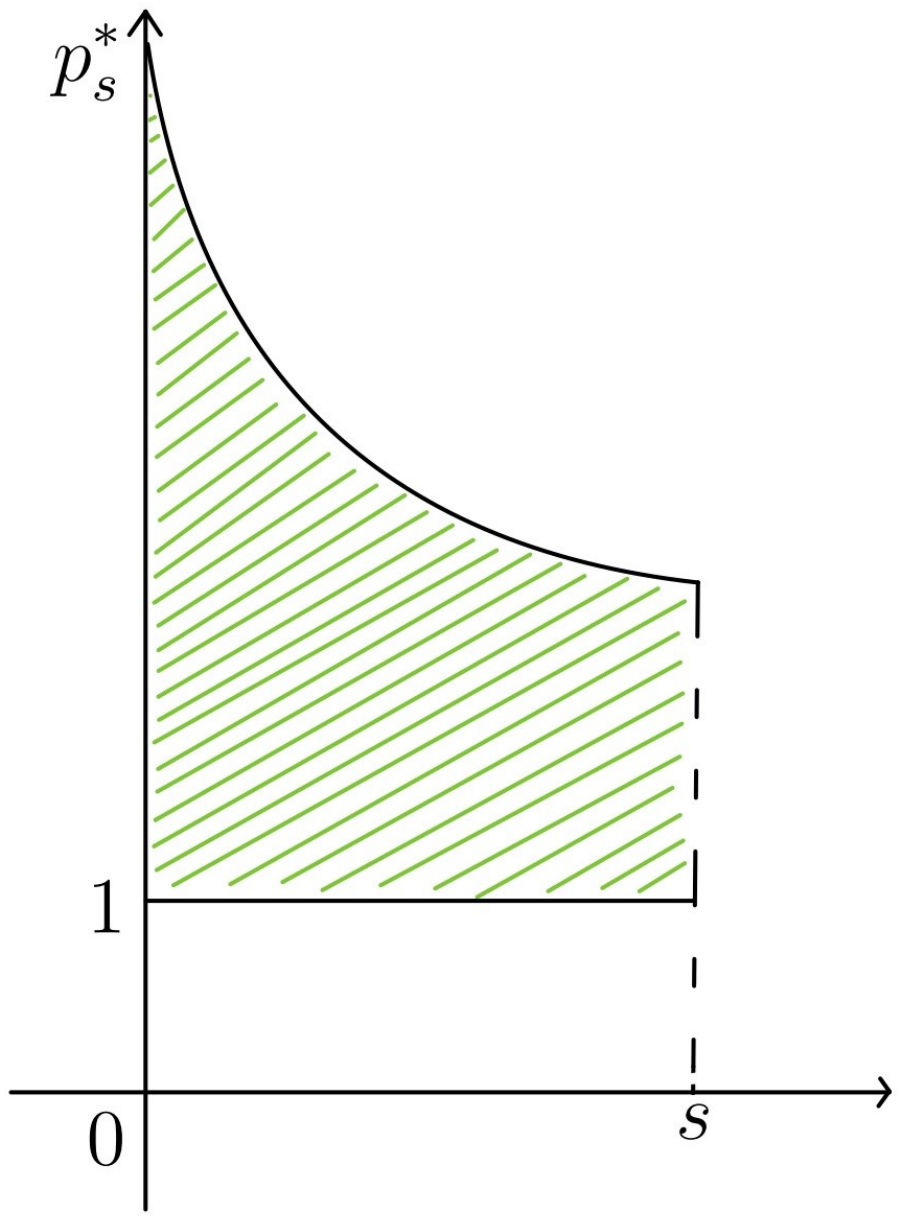}
\end{center}
\caption{The set of points satisfying~\eqref{insiemeimmersioneOmega}.}
\label{green}
\end{figure}

Let us now provide the compact embeddings.  On this matter, we have the following results:

\begin{theorem}\label{teorema3}
Let~$s\in (0,1]$ and~$p\in [1, +\infty)$ be such that~$sp<N$ and~$s\ne p$. Let~$\Omega\subset \R^N$ be an open and bounded domain with Lipschitz boundary. Let~$\widetilde{s}$ and~$\widetilde{p}$ satisfy
\begin{equation}\label{insiemeimmersionecompatta}
\begin{cases}
0\le \widetilde{s}<s,\\
\dfrac{sp}{sp-(p-1)\widetilde{s}} \le\widetilde{p}< 
\frac{Np}{N-(s-\widetilde s)p}.
%%\dfrac{spp^*_s}{sp+(p^*_s-p)\widetilde{s}}.
\end{cases}
\end{equation}

Then, the space~$W^{s,p}(\Omega)$ is compactly embedded in~$W^{\widetilde{s},\widetilde{p}}(\Omega)$.

Moreover, the curve~$\gamma$ defined in~\eqref{curva1} is such that, if~$0< \theta_1 <\theta_2 \le 1$, 
then the space $W^{s_{\theta_1},p_{\theta_1}}(\Omega)$
is compactly embedded in~$W^{s_{\theta_2},p_{\theta_2}}(\Omega)$.
\end{theorem}

\begin{figure}[h]
\begin{center}
\includegraphics[scale=.28]{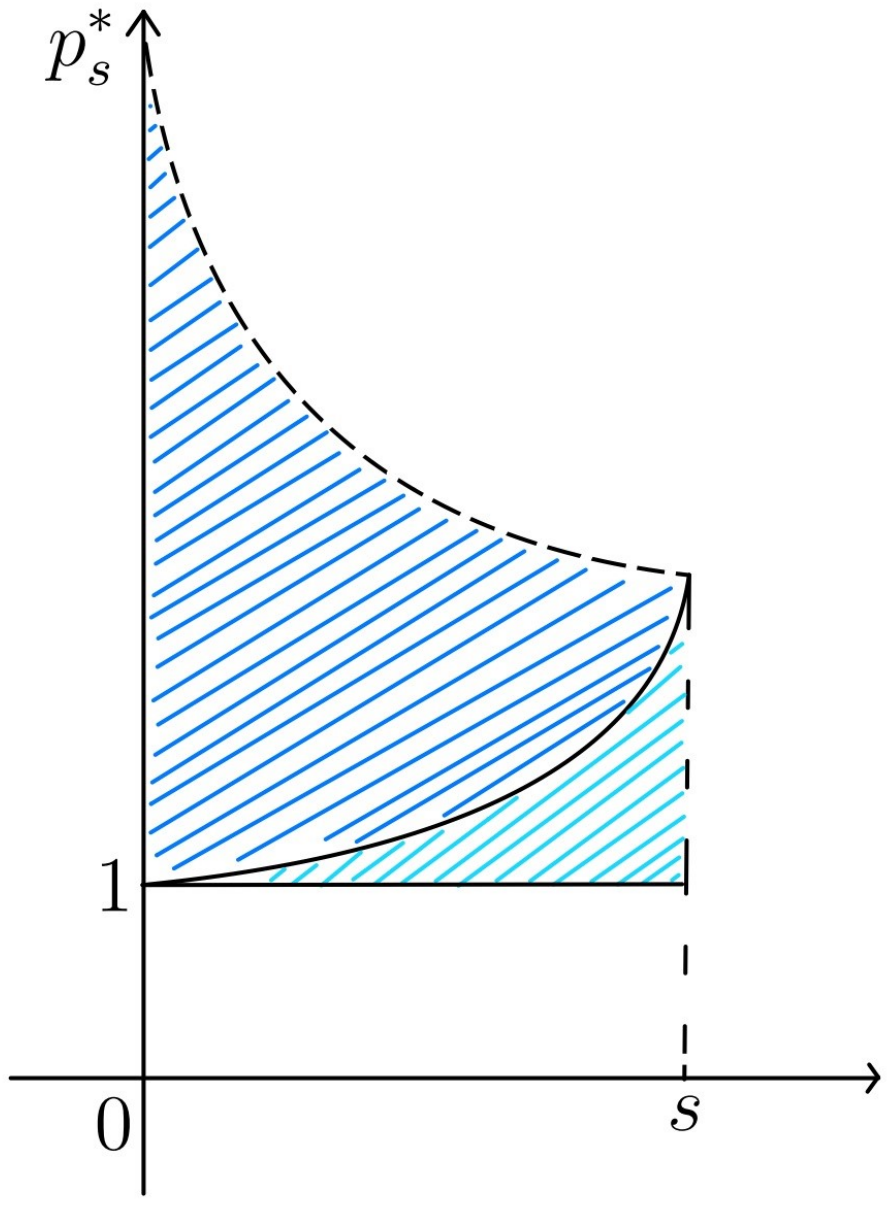}
\end{center}
\caption{The set of points satisfying
\eqref{insiemeimmersionecompatta2}. In dark blue, the set of points
satisfying~\eqref{insiemeimmersionecompatta}.}
\label{blu}
\end{figure}

\begin{corollary}\label{corollariocompattezza}
Let~$s\in (0,1]$ and~$p\in [1, +\infty)$ be such that~$sp<N$ and~$s\ne p$.
Let~$\Omega\subset \R^N$ be an open and bounded domain with Lipschitz boundary.
Let~$\widetilde{s}$ and~$\widetilde{p}$ satisfy
\begin{equation}\label{insiemeimmersionecompatta2}
\begin{cases}
0\le \widetilde{s}<s ,\\
1 \le \widetilde{p}< \frac{Np}{N-(s-\widetilde s)p}.
%%\dfrac{spp^*_s}{sp+(p^*_s-p)\widetilde{s}}.
\end{cases}
\end{equation}

Then, the space~$W^{s,p}(\Omega)$ is compactly embedded in~$W^{\widetilde{s},\widetilde{p}}(\Omega)$.

Moreover, the curve~$\gamma$ defined in~\eqref{curva1} is such that, if~$0< \theta_1 <\theta_2 \le 1$, 
then the space $W^{s_{\theta_1},p_{\theta_1}}(\Omega)$
is compactly embedded in~$W^{s_{\theta_2},p_{\theta_2}}(\Omega)$.
\end{corollary}

The sets of points satisfying~\eqref{insiemeimmersionecompatta}
and~\eqref{insiemeimmersionecompatta2} are depicted in
Figure~\ref{blu}.

\subsection{The case $sp=N$}

In this case, if~$\Omega=\R^N$ we have the following continuous embedding:

\begin{theorem}\label{teorema1sp=N}
Let~$s\in (0,1]$ and~$p\in [1, +\infty)$ be such that~$sp=N$ and~$s\ne p$. Let~$\widetilde{s}$ and~$\widetilde{p}$ satisfy
\begin{equation}\label{insiemeimmersioneRNsp=N}
{\mbox{either}}\quad 
\begin{cases}
0< \widetilde{s}\le s,\\
p \le\widetilde{p}\le \dfrac{N}{\widetilde{s}},
\end{cases}\quad \mbox{ or } \quad
\begin{cases}
\widetilde{s}=0,\\
p\le \widetilde{p}< +\infty.
\end{cases}
\end{equation}

Then, there exists a positive constant~$C=C(N,s,p,\widetilde{s},\widetilde{p})$ such that, for any~$u\in W^{s,p}(\R^N)$,
\begin{equation}\label{inequality1sp=N}
\|u\|_{W^{\widetilde{s},\widetilde{p}}(\R^N)}
\le C\|u\|_{W^{s,p}(\R^N)},
\end{equation}
namely, the space~$W^{s,p}(\R^N)$ is continuously embedded in~$W^{\widetilde{s},\widetilde{p}}(\R^N)$.

Moreover, the curve~$\gamma$ defined in~\eqref{curva1}
is such that, if~$0\le \theta_1 \le \theta_2 \le 1$,
then the space $W^{s_{\theta_1},p_{\theta_1}}(\R^N)$
is continuously embedded in~$W^{s_{\theta_2},p_{\theta_2}}(\R^N)$.

In addition, if~$N=1$, then the space $W^{1,1}(\R^N)$ is also continuously embedded in~$L^\infty(\R^N)$.
\end{theorem}

We point out that, in this case, we can also provide a lower bound for the norm in~\eqref{inequality1sp=N}. Indeed,
by recalling the definition of the~$BMO$ norm in~\cite[Section~5]{MR2240172}, we have that the following result holds true:

\begin{corollary}\label{corollaryBMO}
Let~$s\in (0,1]$ and~$p\in [1, +\infty)$ be such that~$sp=N$ and~$s\ne p$.

Then, for any~$\widetilde{s}$ and~$\widetilde{p}$ satisfying~$\widetilde s\, \widetilde p =N$, there exist two positive constants~$C_1=C_1(N, \widetilde s, \widetilde p)$ and~$C_2=C_2(N, s, p, \widetilde s, \widetilde p)$ such that, for any~$u\in W^{s,p}(\R^N)$,
\begin{equation}\label{inequalityBMO}
\|u\|_{BMO}\le C_1 \|u\|_{W^{\widetilde{s},\widetilde{p}}(\R^N)}\le C_2\|u\|_{W^{s,p}(\R^N)}.
\end{equation}
\end{corollary}

In the case in which~$\Omega\subset \R^N$ is an open and bounded
domain with Lipschitz boundary, our main result reads as follows:

\begin{theorem}\label{teorema2sp=N}
Let~$s\in (0,1]$ and~$p\in [1, +\infty)$ be such that~$sp=N$ and~$s\ne p$.
Let~$\Omega\subset \R^N$ be an open and bounded domain with Lipschitz boundary.
Let~$\widetilde{s}$ and~$\widetilde{p}$ satisfy\footnote{See also footnote~\ref{footnotetrivial}.}
\begin{equation}\label{insiemeimmersioneOmegasp=N}
{\mbox{either}}\quad 
\begin{cases}
0< \widetilde{s}< s,\\
1 \le\widetilde{p}\le \dfrac{N}{\widetilde{s}},
\end{cases}\quad \mbox{ or } \quad
\begin{cases}
\widetilde{s}=0,\\
1\le \widetilde{p}< +\infty,
\end{cases}
\end{equation}

Then, there exists a positive constant~$C=C(N,s,p,\widetilde{s},\widetilde{p},\Omega)$ such that, for any~$u\in W^{s,p}(\Omega)$,
\begin{equation}\label{inequality2sp=N}
\|u\|_{W^{\widetilde{s},\widetilde{p}}(\Omega)}\le C\|u\|_{W^{s,p}(\Omega)},
\end{equation}
namely, the space~$W^{s,p}(\Omega)$ is continuously embedded in~$W^{\widetilde{s},\widetilde{p}}(\Omega)$.

Moreover, the curve~$\gamma$ defined in~\eqref{curva1} is such that, if~$0\le \theta_1 \le \theta_2 \le 1$, 
then the space $W^{s_{\theta_1},p_{\theta_1}}(\Omega)$ is continuously embedded in~$W^{s_{\theta_2},p_{\theta_2}}(\Omega)$.

In addition, if~$N=1$, then the space $W^{1, 1}(\Omega)$ is also continuously embedded in~$L^\infty(\Omega)$.
\end{theorem}

\begin{figure}[h]
\begin{center}
\includegraphics[scale=.24]{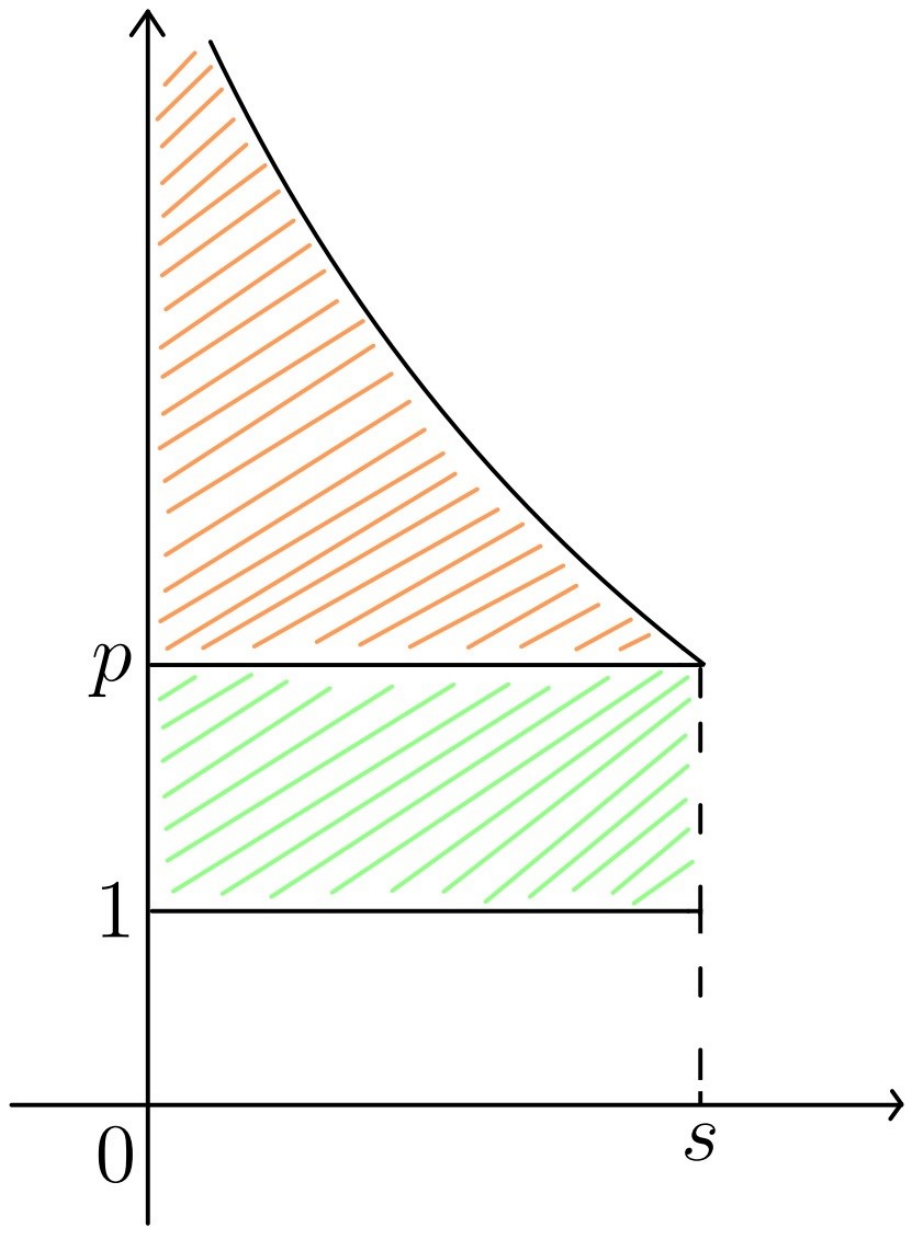}
\end{center}
\caption{In orange the set of points satisfying~\eqref{insiemeimmersioneRNsp=N}. In both orange and green the set of points satisfying~\eqref{insiemeimmersioneOmegasp=N}.
The set of point satisfying~\eqref{2insiemeimmersioneOmegasp=N} coincides with the set of points satisfying~\eqref{insiemeimmersioneOmegasp=N} except for the curve~$\widetilde s\, \widetilde p =N$.}
\label{sp=N}
\end{figure}

Moreover, the following compact embeddings hold true:

\begin{theorem}\label{teorema3sp=N}
Let~$s\in (0,1]$ and~$p\in [1, +\infty)$ be such that~$sp=N$ and~$s\ne p$.
Let~$\Omega\subset \R^N$ be an open and bounded domain with Lipschitz boundary.
Let~$\widetilde{s}$ and~$\widetilde{p}$ satisfy
\begin{equation}\label{2insiemeimmersioneOmegasp=N}
{\mbox{either}}\quad 
\begin{cases}
0< \widetilde{s}< s,\\
1 \le\widetilde{p}< \dfrac{N}{\widetilde{s}},
\end{cases}\quad \mbox{ or } \quad
\begin{cases}
\widetilde{s}=0,\\
1\le \widetilde{p}< +\infty.
\end{cases}
\end{equation}

Then, the space~$W^{s,p}(\Omega)$ is compactly embedded in~$W^{\widetilde{s},\widetilde{p}}(\Omega)$.

Moreover, the curve~$\gamma$ defined in~\eqref{curva1} is such that, if~$0< \theta_1 <\theta_2 \le 1$, 
then the space $W^{s_{\theta_1},p_{\theta_1}}(\Omega)$
is compactly embedded in~$W^{s_{\theta_2},p_{\theta_2}}(\Omega)$.
\end{theorem}

The set of point satisfying~\eqref{insiemeimmersioneRNsp=N},  \eqref{insiemeimmersioneOmegasp=N} and~\eqref{2insiemeimmersioneOmegasp=N} is illustrated in Figure~\ref{sp=N}.

\subsection{The case $sp>N$}
When~$\Omega=\R^N$, our main result reads as follows:

\begin{theorem}\label{teorema1sp>N}
Let~$s\in (0,1]$ and~$p\in [1, +\infty]$ be such that~$sp>N$. Let~$\widetilde{s}$ and~$\widetilde{p}$ satisfy\footnote{We allow~$\widetilde p=+\infty$ in~\eqref{insiemeimmersioneRNsp>N} meaning that~$W^{\widetilde{s}, \infty}(\R^N)= C^{0, \widetilde s}(\R^N)$ (see the comment after Remark~8.3 in~\cite{MR2944369}).}
\begin{equation}\label{insiemeimmersioneRNsp>N}
{\mbox{either }}\quad
\begin{cases}
0\le \widetilde s\le \dfrac{sp-N}{p},\\
p\le \widetilde p\le +\infty,
\end{cases}
\quad\mbox{ or } \quad \begin{cases}
 \dfrac{sp-N}{p}< \widetilde s\le s,\\
p\le \widetilde p\le \dfrac{Np}{N-(s-\widetilde s)p}.
\end{cases}
\end{equation}

Then, there exists a positive constant~$C=C(N,s,p,\widetilde{s},\widetilde{p})$ such that, for any~$u\in W^{s,p}(\R^N)$,
\begin{equation}\label{inequality1sp>N}
\|u\|_{W^{\widetilde{s},\widetilde{p}}(\R^N)}
\leq C\|u\|_{W^{s,p}(\R^N)},
\end{equation}
namely, the space~$W^{s,p}(\R^N)$ is continuously embedded in~$W^{\widetilde{s},\widetilde{p}}(\R^N)$.

Moreover, the curve~$\gamma$ defined in~\eqref{curva1}
is such that, if~$0\le \theta_1 \le \theta_2 \le 1$, 
then the space $W^{s_{\theta_1},p_{\theta_1}}(\R^N)$
is continuously embedded in~$W^{s_{\theta_2},p_{\theta_2}}(\R^N)$.
\end{theorem}

\begin{figure}[h]
\begin{center}
\includegraphics[scale=.28]{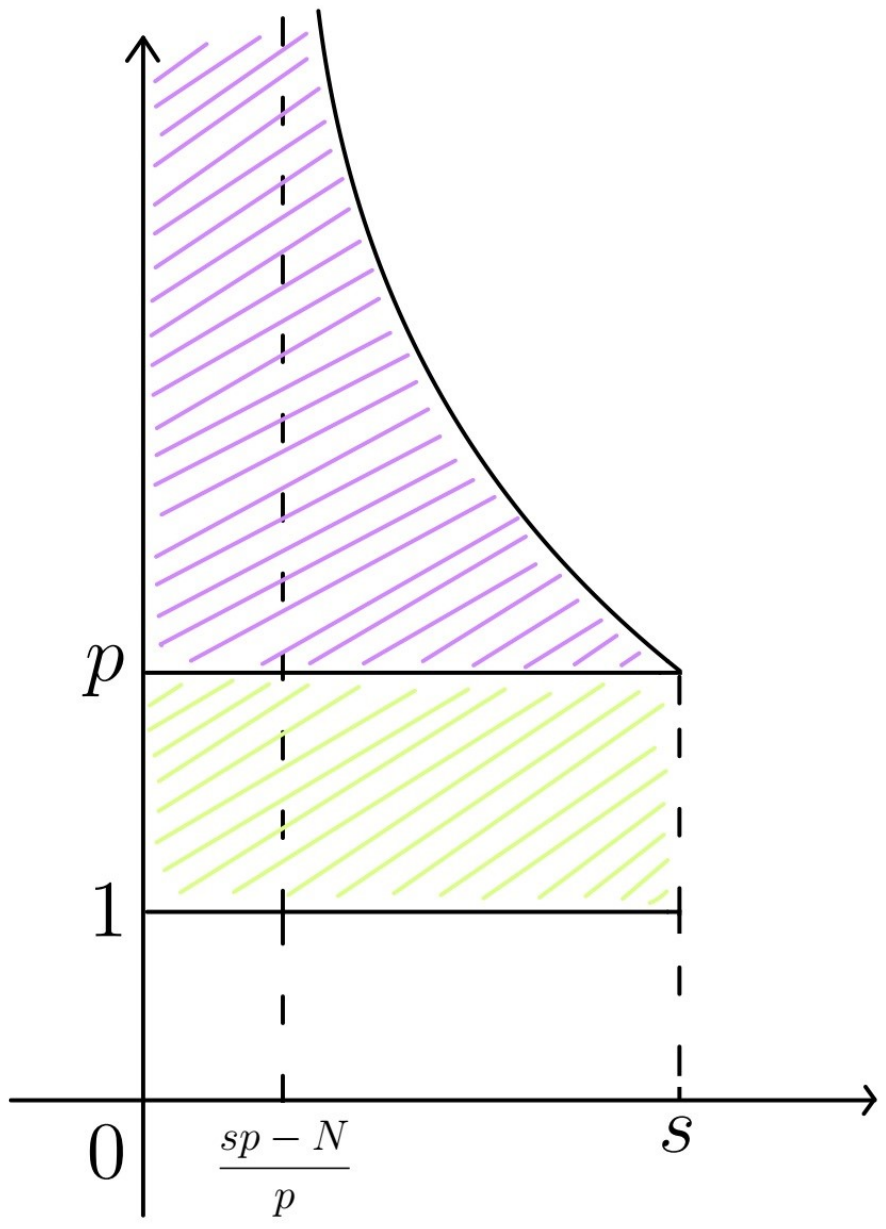}
\end{center}
\caption{In purple the set of points satisfying~\eqref{insiemeimmersioneRNsp>N}. In both purple and yellow the set of points satisfying~\eqref{insiemeimmersioneOmegasp>N}. The set of points satisfying~\eqref{2insiemeimmersioneOmegasp>N} coincides with the set of points satisfying~\eqref{insiemeimmersioneOmegasp>N} except for the curve~$\widetilde p= \dfrac{Np}{N-(s-\widetilde s)p}$.}
\label{sp>N}
\end{figure}

When~$\Omega\subset \R^N$ is an open and bounded domain with Lipschitz boundary, we have the following:

\begin{theorem}\label{teorema2sp>N}
Let~$s\in (0,1]$ and~$p\in [1, +\infty]$ be such that~$sp>N$.
Let~$\Omega\subset \R^N$ be an open and bounded domain with Lipschitz boundary.
Let~$\widetilde{s}$ and~$\widetilde{p}$ satisfy\footnote{See also footnote~\ref{footnotetrivial}.}
\begin{equation}\label{insiemeimmersioneOmegasp>N}
{\mbox{either }}\quad
\begin{cases}
0\le \widetilde s\le \dfrac{sp-N}{p},\\
1\le \widetilde p\le +\infty,
\end{cases}
\quad\mbox{ or } \quad \begin{cases}
 \dfrac{sp-N}{p}< \widetilde s< s,\\
1\le \widetilde p\le \dfrac{Np}{N-(s-\widetilde s)p}.
\end{cases}
\end{equation}

Then, there exists a positive constant~$C=C(N,s,p,\widetilde{s},\widetilde{p}, \Omega)$ such that, for any~$u\in W^{s,p}(\Omega)$,
\begin{equation}\label{2inequality1sp>N}
\|u\|_{W^{\widetilde{s},\widetilde{p}}(\Omega)}
\leq C\|u\|_{W^{s,p}(\Omega)},
\end{equation}
namely, the space~$W^{s,p}(\Omega)$ is continuously embedded in~$W^{\widetilde{s},\widetilde{p}}(\Omega)$.

Moreover, the curve~$\gamma$ defined in~\eqref{curva1}
is such that, if~$0\le \theta_1 \le \theta_2 \le 1$, 
then the space $W^{s_{\theta_1},p_{\theta_1}}(\Omega)$
is continuously embedded in~$W^{s_{\theta_2},p_{\theta_2}}(\Omega)$.
\end{theorem}

The following compact embedding holds true:

\begin{theorem}\label{teorema3sp>N}
Let~$s\in (0,1]$ and~$p\in [1, +\infty]$ be such that~$sp>N$.
Let~$\Omega\subset \R^N$ be an open and bounded domain with Lipschitz boundary. Let~$\widetilde{s}$ and~$\widetilde{p}$ satisfy
\begin{equation}\label{2insiemeimmersioneOmegasp>N}
{\mbox{either }}\quad
\begin{cases}
0\le \widetilde s\le \dfrac{sp-N}{p},\\
1\le \widetilde p\le +\infty,
\end{cases}
\quad\mbox{ or } \quad \begin{cases}
 \dfrac{sp-N}{p}< \widetilde s<s,\\
1\le \widetilde p< \dfrac{Np}{N-(s-\widetilde s)p}.
\end{cases}
\end{equation}

Then, the space~$W^{s,p}(\Omega)$ is compactly embedded in~$W^{\widetilde{s},\widetilde{p}}(\Omega)$.

Moreover, the curve~$\gamma$ defined in~\eqref{curva1}
is such that, if~$0< \theta_1 < \theta_2 \le 1$, 
then the space $W^{s_{\theta_1},p_{\theta_1}}(\Omega)$
is compactly embedded in~$W^{s_{\theta_2},p_{\theta_2}}(\Omega)$.
\end{theorem}

The set of point satisfying~\eqref{insiemeimmersioneRNsp>N}, \eqref{insiemeimmersioneOmegasp>N} and~\eqref{2insiemeimmersioneOmegasp>N} is illustrated in Figure~\ref{sp>N}.

\subsection{Organization of the paper}
The paper is organized as follows. In Section~\ref{optimality} we discuss the assumptions in Theorems~\ref{teorema1}, \ref{teorema2}, \ref{teorema1sp=N}, \ref{teorema2sp=N}, \ref{teorema3sp=N}, \ref{teorema1sp>N}, \ref{teorema2sp>N} and~\ref{teorema3sp>N} and Corollary~\ref{corollariocompattezza}, showing that they are optimal. In Section~\ref{sectionauxiliary} we provide some auxiliary results needed for the proofs of the continuous and compact embeddings. Finally,
in Section~\ref{sectionproofs} we prove our embedding results by addressing separately the cases~$sp<N$, $sp=N$ and~$sp>N$.

\section{Optimality of the embeddings}\label{optimality}

In this section we discuss the optimality of the results stated in this paper. To this end, we will present some preliminary results which show that, with appropriate choices of~$\widetilde s$ and~$\widetilde p$, the space $W^{s,p}(\Omega)$ is not continuously embedded in $W^{\widetilde s,\widetilde p}(\Omega)$, being~$\Omega$ as in~\eqref{Omega}.

Throughout the rest of the section,
we recall that we use the notations $W^{{s}, \infty}(\Omega)=C^{0,{s}}(\Omega)$ and $W^{0,\infty}(\Omega)=L^\infty(\Omega)$.

\subsection{Some scaling properties}
To start with, we provide some useful scaling properties:

\begin{lemma}\label{lemmino0}
Let~$u:\R^N\to\R$ be any measurable function and~$\gamma$, $\beta\in\R$. For any~$\varepsilon\in(0,1)$, let~$v_{\varepsilon, \gamma, \beta}(x):= \varepsilon^\gamma u(\varepsilon^\beta x)$.

Then, for any~$s\in (0,1]$ and~$p\in[1,+\infty)$, 
\begin{equation}\label{lemmino01}
\|v_{\varepsilon, \gamma,\beta}\|_{W^{s, p}(\R^N)} = 
\varepsilon^{\gamma-\frac{\beta N}{p}+\beta s}
\left(\varepsilon^{-\beta sp}\|u\|^p_{L^p(\R^N)} 
+ [u]^p_{W^{s, p}(\R^N)} \right)^{\frac1p}.
\end{equation}

Furthermore, if~$p=+\infty$, for any~$s\in (0,1)$, 
\begin{equation}\label{lemmino01bis}
\|v_{\varepsilon, \gamma,\beta}\|_{W^{s, p}(\R^N)}=
\|v_{\varepsilon, \gamma,\beta}\|_{C^{0, s}(\R^N)}=
\varepsilon^\gamma \|u\|_{L^\infty(\R^N)}+
\varepsilon^{\gamma+\beta s}
\sup_{{x,y\in \R^N}\atop{x\neq y}}\frac{|u(x)-u(y)|}{|x-y|^s}.
\end{equation}
\end{lemma}

\begin{proof}
We first prove~\eqref{lemmino01}. A suitable change of variables leads to
\[
\|v_{\varepsilon, \gamma,\beta}\|^p_{L^p(\R^N)}= \varepsilon^{\gamma p}\int_{\R^N} 
| u(\varepsilon^\beta x)|^p\, dx = \varepsilon^{\gamma p -\beta N}
\|u\|^p_{L^p(\R^N)}
\]
and
\begin{equation}\label{staruno}
[v_{\varepsilon, \gamma,\beta}]^p_{W^{s, p}(\R^N)}
= \varepsilon^{\gamma p} 
\iint_{\R^{2N}} \frac{|u(\varepsilon^\beta x) - u(\varepsilon^\beta y)|^p}{|x-y|^{N+sp}}\, dx \,dy = \varepsilon^{\gamma p -\beta N +\beta sp} [u]^p_{W^{s, p}(\R^N)}.
\end{equation}
Recalling the definition in~\eqref{normaWsp}, we get~\eqref{lemmino01}, as desired.

In order to prove~\eqref{lemmino01bis}, we observe that
\begin{equation}\label{stardue}
\|v_{\varepsilon, \gamma,\beta}\|_{L^\infty(\R^N)}
=\varepsilon^\gamma \|u\|_{L^\infty(\R^N)}
\end{equation}
and 
\begin{equation}\label{startre}\begin{split}
\sup_{{x,y\in \R^N}\atop{x\neq y}}
\frac{|v_{\varepsilon, \gamma,\beta}(x)-v_{\varepsilon, \gamma,\beta}(y)|}{|x-y|^s}
=\;&\varepsilon^{\gamma+\beta s}
\sup_{{x,y\in \R^N}\atop{x\neq y}}\frac{|u(\varepsilon^\beta x)-u(\varepsilon^\beta y)|}{
|\varepsilon^\beta x-\varepsilon^\beta y|^s}
\\= \;&\varepsilon^{\gamma+\beta s}
\sup_{{x,y\in \R^N}\atop{x\neq y}}
\frac{|u(x)-u(y)|}{|x-y|^s},
\end{split}\end{equation}
as desired.
\end{proof}

\begin{lemma}\label{lemmino0Omega}
Let~$\Omega\subset\R^N$ be an open and bounded domain with Lipschitz boundary. Let~$x_0\in \Omega$ and~$R>0$ be such that~$B_{R}(x_0)\subset \Omega$. Let~$\gamma\in\R$.

Let~$u:\R^N\to \R$ be any measurable function such that~$u=0$ in~$\R^N\setminus B_R(x_0)$. 

For any~$\varepsilon\in (0,1)$, let
\[
v_{\varepsilon, \gamma}(x):= \varepsilon^\gamma u\left(x_0+\frac{x}{\varepsilon}\right).
\]

Then, for any~$s\in (0,1]$ and~$p\in[1,+\infty)$, 
\begin{equation}\label{agg00}
\varepsilon^{\gamma p + N - sp} [u]^p_{W^{s, p}(B_R(x_0))}\le
[v_{\varepsilon, \gamma}]^p_{W^{s, p}(\Omega)}
\le \varepsilon^{\gamma p + N - sp} [u]^p_{W^{s, p}(\R^N)}
\end{equation}
and
\begin{equation}\label{lemmino01Omega}
\begin{aligned}
\varepsilon^{\gamma+\frac{N}{p}-s} &\left(\varepsilon^{sp}\|u\|^p_{L^p(\Omega)} + [u]^p_{W^{s, p}(B_R(x_0))} \right)^{\frac1p}
\leq \|v_{\varepsilon, \gamma}\|_{W^{s, p}(\Omega)}  \\
&\leq \varepsilon^{\gamma+\frac{N}{p}-s} \left(\varepsilon^{sp}\|u\|^p_{L^p(\Omega)} + [u]^p_{W^{s, p}(\R^N)} \right)^{\frac1p}.
\end{aligned}
\end{equation}

Furthermore, if~$p=+\infty$, for any~$s\in (0,1)$,
\begin{equation}\label{agg11}
\epsilon^{\gamma-s}
\sup_{{x,y\in B_R(x_0)}\atop{x\neq y}}\frac{\left|u(x)-u(y)\right|}{|x-y|^s}
\le
\sup_{{x,y\in \Omega}\atop{x\ne y}}\frac{|v_{\varepsilon, \gamma}(x)-v_{\varepsilon, \gamma}(y)|}{|x-y|^s}\le
\epsilon^{\gamma-s}
\sup_{{x,y\in \R^N}\atop{x\neq y}}\frac{\left|u(x)-u(y)\right|}{|x-y|^s}
\end{equation}
and
\begin{equation}\label{lemmino01bisOmega}
\begin{aligned}
&\varepsilon^\gamma \|u\|_{L^\infty(\Omega)}+\varepsilon^{\gamma-s}\sup_{\substack{x,y\in B_R \\x\ne y}}\frac{|u(x)-u(y)|}{|x-y|^s}\\
&\qquad\le\|v_{\varepsilon, \gamma}\|_{W^{s, p}(\Omega)}= \|v_{\varepsilon, \gamma}\|_{C^{0, s}(\Omega)}\\
&\qquad\le \varepsilon^\gamma \|u\|_{L^\infty(\Omega)}+\varepsilon^{\gamma-s}\sup_{\substack{x,y\in \R^N \\x\ne y}}\frac{|u(x)-u(y)|}{|x-y|^s}.
\end{aligned}
\end{equation}
\end{lemma}

\begin{proof}
Up to a translation, we can suppose that~$x_0=0$. With this, we have that
\[
\mbox{supp}(v_{\varepsilon, \gamma})
\subseteq B_{\varepsilon R}\subset B_R\subset \Omega.
\]
Therefore,
\begin{equation}\label{zse4rdfrtgyujhio1}
\begin{split}&
\|v_{\varepsilon, \gamma}\|_{L^{ p}(\Omega)}^p=\int_{B_{\varepsilon R}}|v_{\varepsilon, \gamma}(x)|^p\,dx=\epsilon^{\gamma p}\int_{B_{\varepsilon R}}
\left|u\left(\frac{x}{\varepsilon}\right)\right|^p\,dx  \\
&\qquad\qquad=\epsilon^{\gamma p+N}\int_{B_{R}}|u(y)|^p\,dx=\epsilon^{\gamma p+N}\|u\|_{L^p(\Omega)}^p.
\end{split}
\end{equation}
Moreover,
\begin{equation}\label{zse4rdfrtgyujhio2}
\begin{aligned}
[v_{\varepsilon, \gamma}]^p_{W^{s, p}(\Omega)}
&=\iint_{\Omega\times\Omega}\frac{|v_{\varepsilon, \gamma}(x)-v_{\varepsilon, \gamma}(y)|^p}{|x-y|^{N+sp}}\, dx \,dy \\
&=\iint_{B_{\varepsilon R}\times B_{\varepsilon R}}
\frac{|v_{\varepsilon, \gamma}(x)-v_{\varepsilon, \gamma}(y)|^p}{|x-y|^{N+sp}}\, dx \,dy 
+2\iint_{B_{\varepsilon R}\times(\Omega\setminus B_{\varepsilon R})}\frac{|v_{\varepsilon, \gamma}(x)|^p}{|x-y|^{N+sp}}\, dx \,dy  \\
&\ge\varepsilon^{\gamma p} \iint_{B_{\varepsilon R}\times B_{\varepsilon R}}
\frac{\left|u\left(\frac{x}\varepsilon\right) - u\left(\frac{y}\varepsilon\right)\right|^p}{|x-y|^{N+sp}}\, dx \,dy \\
&=
\varepsilon^{\gamma p + N - sp} [u]^p_{W^{s, p}(B_R)}.
\end{aligned}
\end{equation}
Also, by~\eqref{staruno}, used here with~$\beta:=-1$, 
we get
\begin{equation*}
[v_{\varepsilon, \gamma}]^p_{W^{s, p}(\Omega)}
\leq [v_{\varepsilon, \gamma}]^p_{W^{s, p}(\R^N)}
=\varepsilon^{\gamma p + N - sp} [u]^p_{W^{s, p}(\R^N)}.
\end{equation*}
Combining this with~\eqref{zse4rdfrtgyujhio2} we obtain~\eqref{agg00}.

Thus, \eqref{lemmino01Omega}
follows from~\eqref{agg00} and~\eqref{zse4rdfrtgyujhio1}.

Furthermore,
\begin{equation}\label{sup1}
\begin{split}
&\sup_{{x,y\in \Omega}\atop{x\ne y}}\frac{|v_{\varepsilon, \gamma}(x)-v_{\varepsilon, \gamma}(y)|}{|x-y|^s}
=\epsilon^\gamma
\sup_{{x,y\in \Omega}\atop{x\neq y}}\frac{\left|u\left(\frac{x}\varepsilon\right)-u\left(\frac{y}\varepsilon\right)\right|}{|x-y|^s}\ge\epsilon^\gamma
\sup_{{x,y\in B_{\varepsilon R}}\atop{x\neq y}}\frac{\left|u\left(\frac{x}\varepsilon\right)-u\left(\frac{y}\varepsilon\right)\right|}{|x-y|^s}\\
&\qquad\qquad= \epsilon^{\gamma-s}
\sup_{{x,y\in B_R}\atop{x\neq y}}\frac{\left|u(x)-u(y)\right|}{|x-y|^s}.
\end{split}
\end{equation}
Also, from~\eqref{startre} with~$\beta:=-1$, we get
\begin{equation*}
\sup_{{x,y\in \Omega}\atop{x\neq y}}\frac{|v_{\varepsilon, \gamma}(x)-v_{\varepsilon, \gamma}(y)|}{|x-y|^s}\le
\sup_{{x,y\in\R^N}\atop{x\neq y}}\frac{|v_{\varepsilon,\gamma}(x)-v_{\varepsilon,\gamma}(y)|}{|x-y|^s}
=\varepsilon^{\gamma-s}\sup_{\substack{x,y\in\R^N \\x\neq y}}\frac{|u(x)-u(y)|}{|x-y|^s}.
\end{equation*}
{F}rom this and~\eqref{sup1}, we obtain~\eqref{agg11}.

Also, \eqref{lemmino01bisOmega} follows from~\eqref{stardue}
and~\eqref{agg11}.
\end{proof}

\subsection{Towards the optimality statements}

With this preliminary work, we can now establish some optimality results.

\begin{lemma}\label{lemmino1}
Let~$s\in (0,1]$ and~$p\in[1,+\infty)$. Let~$\widetilde{s}$ and~$\widetilde{p}$ satisfy
\begin{equation}\label{parentesilemma1}
\begin{cases}
0\le \widetilde{s}\le 1,\\
1\le\widetilde{p}<p.
\end{cases}
\end{equation}

Then, the space~$W^{s,p}(\R^N)$ is not continuously embedded in $W^{\widetilde s,\widetilde p}(\R^N)$.
\end{lemma}

\begin{proof}
Suppose by contradiction that there exists~$C>0$ such that, for all~$u\in W^{s,p}(\R^N)$,
\begin{equation}\label{f43oytghvdsklhgerio5y4}
\|u\|_{W^{\widetilde s,\widetilde p}(\R^N)}\le C \|u\|_{W^{s,p}(\R^N)}.
\end{equation}

Let~$u\in W^{s,p}(\R^N)\setminus\{0\}$ and, for any~$\varepsilon\in(0,1)$, define the function
\[
v_\varepsilon(x):=\varepsilon^{\frac{N}{p}}u(\varepsilon x).
\]
By~\eqref{lemmino01} (used here with~$\gamma:=N/p$ and~$\beta:=1$), we get that
\[
\|v_\varepsilon\|_{W^{s,p}(\R^N)}=\left(\|u\|_{L^p(\R^N)}^p+\varepsilon^{sp}\,[u]_{W^{s,p}(\R^N)}^p \right)^\frac{1}{p}
\]
and
\[
\|v_\varepsilon\|_{W^{\widetilde s,\widetilde p}(\R^N)}=
\varepsilon^{-N\left(\frac{1}{\widetilde p}-\frac{1}{p}\right)}
\left(\|u\|_{L^{\widetilde p}(\R^N)}^{\widetilde p}+\varepsilon^{\widetilde s \widetilde p}\,[u]_{W^{\widetilde s,\widetilde p}(\R^N)}^{\widetilde p} \right)^\frac{1}{\widetilde p}.
\]

Thus, plugging~$v_\varepsilon$ into~\eqref{f43oytghvdsklhgerio5y4}, we obtain that
\[
\varepsilon^{-N\left(\frac{1}{\widetilde p}-\frac{1}{p}\right)}
\left(\|u\|_{L^{\widetilde p}(\R^N)}^{\widetilde p}+\varepsilon^{\widetilde s \widetilde p}\,[u]_{W^{\widetilde s,\widetilde p}(\R^N)}^{\widetilde p} \right)^\frac{1}{\widetilde p}\le C \left(\|u\|_{L^p(\R^N)}^p+\varepsilon^{sp}\,[u]_{W^{s,p}(\R^N)}^p \right)^\frac{1}{p}.
\]
Taking the limit as~$\varepsilon\searrow0$ and recalling that~$\widetilde p<p$, we obtain the desired contradiction.
\end{proof}

We now address separately the cases~$sp<N$, $sp=N$ and~$sp>N$.  When~$sp<N$ and~$\Omega$ is as in~\eqref{Omega}, we have the following:

\begin{lemma}\label{lemmino2}
Let~$s\in (0,1]$ and~$p\in[1,+\infty)$ be such that~$sp<N$. Let~$\Omega$ be as in~\eqref{Omega} and assume that~$\widetilde s$ and~$\widetilde p$ satisfy
\begin{equation}\label{parentesilemma2}
\mbox{either}\quad
\begin{cases}
s<\widetilde s \le 1,\\
p\le\widetilde p\le +\infty,
\end{cases}\quad\mbox{ or }\quad \begin{cases}
0\le\widetilde s \le s,\\
%%\dfrac{s p p^*_s}{sp + (p^*_s -p)\widetilde s}
\frac{Np}{N-(s-\widetilde s)p}<\widetilde p\le +\infty.
\end{cases}
\end{equation}

Then, the space $W^{s,p}(\Omega)$ is not continuously embedded in $W^{\widetilde s,\widetilde p}(\Omega)$.
\end{lemma}

\begin{proof}
We first deal with the case~$\Omega=\R^N$.
Also, we first consider the case~$\widetilde p = +\infty$.
Suppose by contradiction that there exists~$C>0$ such that, for all~$u\in W^{s,p}(\R^N)$,
\begin{equation}\label{f43oytghvdsklhgerio5y42sdfgh}
\|u\|_{C^{0,\widetilde{s}}(\R^N)}\le C \|u\|_{W^{s,p}(\R^N)}.
\end{equation}
Let~$u\in W^{s,p}(\R^N)\setminus\{0\}$ and, 
for any~$\varepsilon\in(0,1)$, define the function
\begin{equation}\label{zxdfghjk-}
v_\varepsilon(x):=\varepsilon^{-\frac{N-sp}{p}}u\left(\frac{x}{\varepsilon}\right).
\end{equation}
Then, by~\eqref{lemmino01} and~\eqref{lemmino01bis} (used
here with~$\gamma:=-(N-sp)/p$ and~$\beta:=-1$), 
we have that
\begin{equation}\label{xsdrtgvbnhjiol}
\|v_\varepsilon\|_{W^{s,p}(\R^N)}=\left(\varepsilon^{sp}\,\|u\|_{L^p(\R^N)}^p+ [u]_{W^{s,p}(\R^N)}^p \right)^\frac{1}{p}
\end{equation}
and
\[
\|v_{\varepsilon}\|_{C^{0, \widetilde s}(\R^N)}=
\varepsilon^{-\frac{N-sp}p} \|u\|_{L^\infty(\R^N)}+
\varepsilon^{-\frac{N-sp}p-\widetilde s}\sup_{\substack{x,y\in \R^N \\x\neq y}}\frac{|u(x)-u(y)|}{|x-y|^{\widetilde s}}.
\]
As a result, plugging~$v_{\varepsilon}$ into~\eqref{f43oytghvdsklhgerio5y42sdfgh},
we obtain that
\[
\varepsilon^{-\frac{N-sp}p}\left(\|u\|_{L^\infty(\R^N)}+
\varepsilon^{-\widetilde s}\sup_{\substack{x,y\in \R^N \\x\neq y}}\frac{|u(x)-u(y)|}{|x-y|^{\widetilde s}} \right)\leq C\left(\varepsilon^{sp}\,\|u\|_{L^p(\R^N)}^p+ [u]_{W^{s,p}(\R^N)}^p \right)^\frac{1}{p}.
\]
Thus, taking the limit as~$\varepsilon\searrow0$ we get a contradiction.

Now, we assume~$\widetilde p \ne +\infty$.
Suppose by contradiction that there exists~$C>0$ such that, for all~$u\in W^{s,p}(\R^N)$,
\begin{equation}\label{f43oytghvdsklhgerio5y42}
\|u\|_{W^{\widetilde s,\widetilde p}(\R^N)}\le C \|u\|_{W^{s,p}(\R^N)}.
\end{equation}

Let~$u\in W^{s,p}(\R^N)\setminus\{0\}$ and, 
for any~$\varepsilon\in(0,1)$, define the function~$v_\varepsilon$ as in~\eqref{zxdfghjk-}.
By~\eqref{lemmino01} (used here~$\gamma :=-(N-sp)/p$ and~$\beta:=-1$), we infer that
\[
\|v_\varepsilon\|_{W^{\widetilde s,\widetilde p}(\R^N)}=\varepsilon^{-\frac{N}{p} +\frac{N}{\widetilde p} + s-\widetilde s}
\left(\varepsilon^{\widetilde s \widetilde p}\,\|u\|_{L^{\widetilde p}(\R^N)}^{\widetilde p}+ [u]_{W^{\widetilde s,\widetilde p}(\R^N)}^{\widetilde p} \right)^\frac{1}{\widetilde p}.
\]
Thus, plugging~$v_\varepsilon$ into~\eqref{f43oytghvdsklhgerio5y42}
and recalling also~\eqref{xsdrtgvbnhjiol},
\[
\varepsilon^{-\frac{N}{p} +\frac{N}{\widetilde p} + s-\widetilde s}\left(\varepsilon^{\widetilde s \widetilde p}\,\|u\|_{L^{\widetilde p}(\R^N)}^{\widetilde p}+ [u]_{W^{\widetilde s,\widetilde p}(\R^N)}^{\widetilde p} \right)^\frac{1}{\widetilde p}\le C\left(\varepsilon^{sp}\,\|u\|_{L^p(\R^N)}^p+[u]_{W^{s,p}(\R^N)}^p \right)^\frac{1}{p}.
\]
We notice that, in both the cases in~\eqref{parentesilemma2},
\begin{equation}\label{esponentenegativo}
-\frac{N}{p} +\frac{N}{\widetilde p} + s-\widetilde s<0.
\end{equation}
Hence, taking the limit as~$\varepsilon\searrow0$, we obtain the desired contradiction.

Let now~$\Omega\subset\R^N$ be an open bounded domain with Lipschitz boundary.
Up to a translation, we can assume that~$0\in \Omega$. Let~$R>0$ be such that~$B_R\subset \Omega$. We also take~$u:\R^N\to \R$ measurable such that~$u=0$ in~$\R^N\setminus B_R$, $u$ does not vanish identically and~$[u]_{W^{s, p}(\R^N)}<+\infty$.
For any~$\varepsilon\in (0,1)$, we define
\[
v_{\varepsilon}(x):= \varepsilon^{-(N-sp)/p} u\left(\frac{x}{\varepsilon}\right).
\]
Suppose by contradiction that there exists~$C>0$ such that, for all~$u\in W^{s,p}(\Omega)$, 
\[
\|u\|_{W^{\widetilde s,\widetilde p}(\Omega)}\le C \|u\|_{W^{s,p}(\Omega)}.
\]

We exploit Lemma~\ref{lemmino0Omega} with~$\gamma:= -(N-sp)/p$. More precisely, if~$\widetilde p\ne +\infty$, we use~\eqref{lemmino01Omega} to find that
\[
\|v_\varepsilon\|_{W^{\widetilde s,\widetilde p}(\Omega)}\ge \varepsilon^{-\frac{N}{p} +\frac{N}{\widetilde p} + s-\widetilde s}\left(\varepsilon^{\widetilde s \widetilde p}\,\|u\|_{L^{\widetilde p}(\Omega)}^{\widetilde p}+ [u]_{W^{\widetilde s,\widetilde p}(B_R)}^{\widetilde p} \right)^\frac{1}{\widetilde p}
\]
and
\[
\|v_\varepsilon\|_{W^{\widetilde s,\widetilde p}(\Omega)}\le C \|v_\varepsilon\|_{W^{s,p}(\Omega)}\le C\left(\varepsilon^{sp}\,\|u\|_{L^p(\Omega)}^p+[u]_{W^{s,p}(\R^N)}^p \right)^\frac{1}{p}.
\]
The last two displays entail that
\[
\varepsilon^{-\frac{N}{p} +\frac{N}{\widetilde p} + s-\widetilde s}\left(\varepsilon^{\widetilde s \widetilde p}\,\|u\|_{L^{\widetilde p}(\Omega)}^{\widetilde p}+ [u]_{W^{\widetilde s,\widetilde p}(B_R)}^{\widetilde p} \right)^\frac{1}{\widetilde p}\le C\left(\varepsilon^{sp}\,\|u\|_{L^p(\Omega)}^p+[u]_{W^{s,p}(\R^N)}^p \right)^\frac{1}{p}.
\]
Recalling~\eqref{esponentenegativo} and taking the limit as~$\varepsilon\searrow 0$, we get the desired contradiction.

If instead~$\widetilde p = +\infty$, we exploit~\eqref{lemmino01Omega} and~\eqref{lemmino01bisOmega} and see that
\[
\begin{split}
&\varepsilon^{-\frac{N-sp}{p}} \|u\|_{L^\infty(\Omega)}+\varepsilon^{-\frac{N-sp}{p}-\widetilde s}\sup_{\substack{x,y\in B_R \\x\ne y}}\frac{|u(x)-u(y)|}{|x-y|^{\widetilde s}}\le \|v_\varepsilon\|_{C^{0, \widetilde s}(\Omega)}\\
&\qquad\quad\le C \|v_\varepsilon\|_{W^{s,p}(\Omega)}\le C\left(\varepsilon^{sp}\,\|u\|_{L^p(\Omega)}^p+[u]_{W^{s,p}(\R^N)}^p \right)^\frac{1}{p}.
\end{split}
\]
By taking the limit as~$\varepsilon\searrow 0$, we get the desired contradiction also in the case~$\widetilde p =+\infty$.
\end{proof}

We now consider the case~$sp=N$. In this case, the following result holds true:

\begin{lemma}\label{lemmino3}
Let~$s\in (0,1]$ and~$p\in[1,+\infty)$ be such that~$sp=N$. Let~$\Omega$ be as in~\eqref{Omega}. 

If~$N\geq 2$, assume that~$\widetilde s$ and~$\widetilde p$ satisfy\footnote{With a slight abuse of notation, when~$\widetilde s=0$ in
the second possibility in~\eqref{parentesilemma3}, we allow~$\widetilde p=+\infty$.}
\begin{equation}\label{parentesilemma3}
\mbox{either}\quad
\begin{cases}
s<\widetilde s \le 1,\\
p\le\widetilde p\le +\infty,
\end{cases}\quad\mbox{ or }\quad \begin{cases}
0\le\widetilde s \le s,\\
\dfrac{N}{\widetilde s}<\widetilde p\le +\infty.
\end{cases}
\end{equation}

If~$N=1$, assume that~$\widetilde s$ and~$\widetilde p$ satisfy
\begin{equation}\label{parentesilemma3BIS}
\mbox{either}\quad
\begin{cases}
s<\widetilde s \le 1,\\
p\le\widetilde p\le +\infty,
\end{cases}\quad\mbox{ or }\quad \begin{cases}
0< \widetilde s \le s,\\
\dfrac{N}{\widetilde s}<\widetilde p\le +\infty.
\end{cases}
\end{equation}

Then, the space $W^{s,p}(\Omega)$ is not continuously embedded in $W^{\widetilde s,\widetilde p}(\Omega)$.
\end{lemma}

We point out that the only difference between the cases~$N\ge2$
in~\eqref{parentesilemma3} and~$N=1$ in~\eqref{parentesilemma3BIS}
is that in the latter we do not allow~$\widetilde s=0$.
The reason for this is that when~$N=1$ and~$\widetilde s=0$,
from~\eqref{parentesilemma3BIS} we would have that~$\widetilde p=+\infty$, namely we look at the embedding of~$W^{s,p}(\Omega)$
into~$L^\infty(\Omega)$ when~$sp=1$.
In this case, when~$p\in(1,+\infty)$ we know that this embedding is false by~\cite[formulas~(1.9)--(1.10) in
Theorem~B]{MR3990737}, while when~$p=1$ 
there is a continuous embedding of~$W^{1,1}(\Omega)$ into~$L^\infty(\Omega)$, in light of~\cite[formula~(1.6) in
Theorem~B]{MR3990737}.

\begin{proof}[Proof of Lemma~\ref{lemmino3}]
We first consider the case~$\widetilde p = +\infty$.
We recall that the choice of~$\widetilde s=0$
can occur only in~\eqref{parentesilemma3} and leads to $W^{\widetilde s, \widetilde p}(\Omega)=W^{0, \infty}(\Omega)=L^\infty(\Omega)$.
In this case, we have that
if~$s\in (0,1]$ and~$p\in(1,+\infty)$ with~$sp=N$, then
\begin{equation*}%\label{inappendice}
\mbox{if $N\ge 2$, then $W^{s, p}(\Omega)$ is not continuously embedded in $L^\infty(\Omega)$},
\end{equation*}
%%However, not to interrupt the flow of the argument, we postpone the proof of~\eqref{inappendice} to Appendix~\ref{appendixcounterexample}.
as established in~\cite[formulas~(1.9)--(1.10) in
Theorem~B]{MR3990737} (see also Appendix~A
of~\cite{MR4053239} for additional explicit examples).

Hence, from now on, we suppose that~$\widetilde s>0$.

We first address the case~$\Omega=\R^N$.
If~$\widetilde p =+\infty$, we suppose by contradiction that there exists~$C>0$ such that, for all~$u\in W^{s,p}(\R^N)$,
\begin{equation}\label{f43oytghvdsklhgerio5y43ihujok}
\|u\|_{C^{0, \widetilde s}(\R^N)}\le C \|u\|_{W^{s,p}(\R^N)}.
\end{equation}
For any~$\varepsilon\in (0, 1)$ we define the function
\begin{equation}\label{opjihbjkhv}
v_\varepsilon(x):=u\left(\frac{x}{\varepsilon}\right).
\end{equation}
By~\eqref{lemmino01} and~\eqref{lemmino01bis} (used here with~$\gamma:=0$ and~$\beta:=-1$), we have that
\begin{equation}\label{yguhijnokm}
\|v_\varepsilon\|_{W^{s,p}(\R^N)}=\left(\varepsilon^{sp}\,\|u\|_{L^p(\R^N)}^p+ [u]_{W^{s,p}(\R^N)}^p \right)^\frac{1}{p}
\end{equation}
and
\[
\|v_{\varepsilon, \gamma}\|_{C^{0, \widetilde s}(\R^N)}= \|u\|_{L^\infty(\R^N)}+
\varepsilon^{-\widetilde s}\sup_{\substack{x,y\in \R^N \\x\neq y}}\frac{|u(x)-u(y)|}{|x-y|^{\widetilde s}}.
\]
Thus, plugging~$v_\varepsilon$ into~\eqref{f43oytghvdsklhgerio5y43ihujok}, we get
\[
\|u\|_{L^\infty(\R^N)}+\varepsilon^{-\widetilde s}\sup_{\substack{x,y\in \R^N \\x\neq y}}\frac{|u(x)-u(y)|}{|x-y|^{\widetilde s}} \le C\left(\varepsilon^{sp}\,\|u\|_{L^p(\R^N)}^p+ [u]_{W^{s,p}(\R^N)}^p \right)^\frac{1}{p}.
\]
Since~$\widetilde s>0$, by taking the limit as~$\varepsilon\searrow0$, we get a contradiction with~\eqref{f43oytghvdsklhgerio5y43ihujok}.

We now assume that~$\widetilde p \ne +\infty$.
Suppose by contradiction that there exists~$C>0$ such that, for all~$u\in W^{s,p}(\R^N)$,
\begin{equation}\label{f43oytghvdsklhgerio5y43}
\|u\|_{W^{\widetilde s,\widetilde p}(\R^N)}\le C \|u\|_{W^{s,p}(\R^N)}.
\end{equation}

For any~$\varepsilon\in (0, 1)$ we consider the function~$v_\varepsilon$ as in~\eqref{opjihbjkhv}.
By~\eqref{lemmino01} (used here with~$\gamma:=0$ and~$\beta:=-1$),
we infer that~\eqref{yguhijnokm} holds and 
\[
\|v_\varepsilon\|_{W^{\widetilde s,\widetilde p}(\R^N)}=\varepsilon^{-\left(\widetilde s-\frac{N}{\widetilde p}\right)}\left(\varepsilon^{\widetilde s \widetilde p}\|u\|_{L^{\widetilde p}(\R^N)}^{\widetilde p}+ [u]_{W^{\widetilde s,\widetilde p}(\R^N)}^{\widetilde p} \right)^\frac{1}{\widetilde p}.
\]
Thus, plugging~$v_\varepsilon$ into~\eqref{f43oytghvdsklhgerio5y43}, we get
\begin{equation}\label{6t574r832gvhfjcdxbvmfncdploikujyhythn:fbrek}
\varepsilon^{-\left(\widetilde s-\frac{N}{\widetilde p}\right)}
\left(\varepsilon^{\widetilde s \widetilde p}\,\|u\|_{L^{\widetilde p}(\R^N)}^{\widetilde p}+ [u]_{W^{\widetilde s,\widetilde p}(\R^N)}^{\widetilde p} \right)^\frac{1}{\widetilde p}\le C
\left(\varepsilon^{sp}\,\|u\|_{L^p(\R^N)}^p+[u]_{W^{s,p}(\R^N)}^p \right)^\frac{1}{p}.
\end{equation}
Now, in both the cases in~\eqref{parentesilemma3} when~$N\ge2$ and in~\eqref{parentesilemma3BIS} when~$N=1$,
we have that
\begin{equation}\label{esponentepositivo}
\widetilde s-\frac{N}{\widetilde p}>0.
\end{equation}
Hence, taking the limit as~$\varepsilon\searrow0$
in~\eqref{6t574r832gvhfjcdxbvmfncdploikujyhythn:fbrek}, we obtain the desired contradiction.

Let now~$\Omega\subset\R^N$ be an open and bounded domain with Lipschitz boundary.
Up to a translation, we can assume that~$0\in \Omega$.
Let~$R>0$ be such that~$B_R\subset \Omega$ and take~$u:\R^N\to \R$ measurable such that~$u=0$ in~$\R^N\setminus B_R$, $u$ does not vanish identically and~$[u]_{W^{s, p}(\R^N)}<+\infty$.
For any~$\varepsilon\in (0,1)$, we define
\[
v_{\varepsilon}(x):= u\left(\frac{x}{\varepsilon}\right).
\]
Let~$\widetilde s >0$ and suppose by contradiction that there exists~$C>0$ such that, for all~$u\in W^{s,p}(\Omega)$, 
\[
\|u\|_{W^{\widetilde s,\widetilde p}(\Omega)}\le C \|u\|_{W^{s,p}(\Omega)}.
\]
We exploit Lemma~\ref{lemmino0Omega} with~$\gamma:= 0$.
In this way, if~$\widetilde p\ne +\infty$, using~\eqref{lemmino01Omega}
we have that
\[
\|v_\varepsilon\|_{W^{\widetilde s,\widetilde p}(\Omega)}\ge \varepsilon^{-\left(\widetilde s-\frac{N}{\widetilde p}\right)}
\left(\varepsilon^{\widetilde s \widetilde p}\,\|u\|_{L^{\widetilde p}(\Omega)}^{\widetilde p}+ [u]_{W^{\widetilde s,\widetilde p}(B_R)}^{\widetilde p} \right)^\frac{1}{\widetilde p}
\]
and
\[
\|v_\varepsilon\|_{W^{\widetilde s,\widetilde p}(\Omega)}\le C \|v_\varepsilon\|_{W^{s,p}(\Omega)}\le C\left(\varepsilon^{sp}\,\|u\|_{L^p(\Omega)}^p+[u]_{W^{s,p}(\R^N)}^p \right)^\frac{1}{p}.
\]
The last two displays show that
\[
\varepsilon^{-\left(\widetilde s-\frac{N}{\widetilde p}\right)}
\left(\varepsilon^{\widetilde s \widetilde p}\,\|u\|_{L^{\widetilde p}(\Omega)}^{\widetilde p}+ [u]_{W^{\widetilde s,\widetilde p}(B_R)}^{\widetilde p} \right)^\frac{1}{\widetilde p}\le C\left(\varepsilon^{sp}\,\|u\|_{L^p(\Omega)}^p+[u]_{W^{s,p}(\R^N)}^p \right)^\frac{1}{p}.
\]
Recalling~\eqref{esponentepositivo} and taking the limit as~$\varepsilon\searrow 0$, we get the desired contradiction.

If instead~$\widetilde p = +\infty$,
we use~\eqref{lemmino01bisOmega} and we find that
\[
\begin{split}
&\|u\|_{L^\infty(\Omega)}+\varepsilon^{-\widetilde s}\sup_{\substack{x,y\in B_R \\x\ne y}}\frac{|u(x)-u(y)|}{|x-y|^{\widetilde s}}\le \|v_\varepsilon\|_{C^{0, \widetilde s}(\Omega)}\\
&\qquad\le C \|v_\varepsilon\|_{W^{s,p}(\Omega)}\le C\left(\varepsilon^{sp}\,\|u\|_{L^p(\Omega)}^p+[u]_{W^{s,p}(\R^N)}^p \right)^\frac{1}{p}.
\end{split}
\]
Taking the limit as~$\varepsilon\searrow 0$, we get the desired contradiction.
\end{proof}

For~$sp>N$, we provide the following:

\begin{lemma}\label{lemmino4}
Let~$s\in (0,1]$ and~$p\in[1,+\infty]$ be such that~$sp>N$. Let~$\Omega$ be as in~\eqref{Omega} and assume that~$\widetilde s$ and~$\widetilde p$ satisfy
\begin{equation}\label{parentesilemma4}
\mbox{either}\quad
\begin{cases}
s<\widetilde s \le 1,\\
p\le\widetilde p\le +\infty,
\end{cases}\quad\mbox{ or }\quad \begin{cases}
\dfrac{sp-N}{p}<\widetilde s \le s,\\
\dfrac{Np}{N-(s-\widetilde s)p}<\widetilde p\le +\infty.
\end{cases}
\end{equation}

Then, the space $W^{s,p}(\Omega)$ is not continuously embedded in $W^{\widetilde s,\widetilde p}(\Omega)$.
\end{lemma}

\begin{proof}
We first address the case~$\Omega=\R^N$.
If~$\widetilde p \ne +\infty$, we assume by contradiction that there exists~$C>0$ such that, for all~$u\in W^{s,p}(\R^N)$,
\begin{equation}\label{f43oytghvdsklhgerio5y44}
\|u\|_{W^{\widetilde s,\widetilde p}(\R^N)}\le C \|u\|_{W^{s,p}(\R^N)}.
\end{equation}

For any~$\varepsilon\in(0,1)$ we define the function
\begin{equation}\label{vvarepsilon4}
v_\varepsilon(x):=\varepsilon^{s-\frac{N}{p}} u\left(\frac{x}{\varepsilon}\right).
\end{equation}
By~\eqref{lemmino01} (used here with~$\gamma:=s-N/p$ and~$\beta:=-1$), we infer that
\[
\|v_\varepsilon\|_{W^{s,p}(\R^N)}=\left(\varepsilon^{sp}\,\|u\|_{L^p(\R^N)}^p+ [u]_{W^{s,p}(\R^N)}^p \right)^\frac{1}{p}
\]
and
\[
\|v_\varepsilon\|_{W^{\widetilde s,\widetilde p}(\R^N)}=\varepsilon^{-\frac{N}{p}+\frac{N}{\widetilde p}+s-\widetilde s}\left(\varepsilon^{\widetilde s \widetilde p}\|u\|_{L^{\widetilde p}(\R^N)}^{\widetilde p}+ [u]_{W^{\widetilde s,\widetilde p}(\R^N)}^{\widetilde p} \right)^\frac{1}{\widetilde p}.
\]
Thus, plugging $v_\varepsilon$ into~\eqref{f43oytghvdsklhgerio5y44}, we get
\[
\varepsilon^{-\frac{N}{p}+\frac{N}{\widetilde p}+s-\widetilde s}
\left(\varepsilon^{\widetilde s \widetilde p}\,\|u\|_{L^{\widetilde p}(\R^N)}^{\widetilde p}+ [u]_{W^{\widetilde s,\widetilde p}(\R^N)}^{\widetilde p} \right)^\frac{1}{\widetilde p}\le C
\left(\varepsilon^{sp}\,\|u\|_{L^p(\R^N)}^p+[u]_{W^{s,p}(\R^N)}^p \right)^\frac{1}{p}.
\]
We notice that, in both the cases in~\eqref{parentesilemma4},
\begin{equation}\label{esponentenegativo2}
-\frac{N}{p}+\frac{N}{\widetilde p}+s-\widetilde s<0.
\end{equation}
Hence, taking the limit as~$\varepsilon\searrow 0$, we obtain the desired contradiction.

We now address the case $\widetilde p=+\infty$.  Assume by contradiction that there exists~$C>0$ such that, for all~$u\in W^{s,p}(\R^N)$,
\begin{equation}\label{f43oytghvdsklhgerio5y45}
\|u\|_{C^{0, \widetilde s}(\R^N)}\le C \|u\|_{W^{s,p}(\R^N)}.
\end{equation}
For any $\varepsilon\in(0,1)$, we recall the definition in~\eqref{vvarepsilon4} and we exploit~\eqref{lemmino01bis}
(with~$\gamma:=s-N/p$ and~$\beta:=-1$)
to see that
\[
\|v_\varepsilon\|_{C^{0, \widetilde s}(\R^N)} =\varepsilon^{s-\frac{N}{p}} \|u\|_{L^\infty(\R^N)} + \varepsilon^{s-\widetilde s -\frac{N}{p}}  \sup_{\substack{x,y\in\R^N\\x\ne y}}\dfrac{|u(x) - u(y)|}{|x-y|^{\widetilde s}}.
\]
Thus, using also~\eqref{lemmino01} and plugging~$v_\varepsilon$ into~\eqref{f43oytghvdsklhgerio5y45}, we obtain that
\[
\varepsilon^{s-\frac{N}{p}} \|u\|_{L^\infty(\R^N)} + \varepsilon^{s-\widetilde s -\frac{N}{p}}  \sup_{\substack{x,y\in\R^N\\x\ne y}}\dfrac{|u(x) - u(y)|}{|x-y|^{\widetilde s}}\le C
\left(\varepsilon^{sp}\,\|u\|_{L^p(\R^N)}^p+[u]_{W^{s,p}(\R^N)}^p \right)^\frac{1}{p}.
\]
We notice that, in both the cases in~\eqref{parentesilemma4},
\begin{equation}\label{esponentenegativo3}
\widetilde s> s-\frac{N}{p}.
\end{equation}
Hence, taking the limit as~$\varepsilon\searrow 0$, we obtain the desired contradiction.

Now, let~$\Omega\subset\R^N$ be any open bounded domain with Lipschitz boundary.
Up to a translation, we can assume~$0\in \Omega$.
Let~$R>0$ be such that~$B_R\subset \Omega$ and take~$u:\R^N\to \R$ measurable such that~$u=0$ in $\R^N\setminus B_R$,
$u$ does not vanish identically and~$[u]_{W^{s,p}(\R^N)}<+\infty$.
For any~$\varepsilon\in (0,1)$, let
\[
v_{\varepsilon}(x):= \varepsilon^{s-\frac{N}{p}} u\left(\frac{x}{\varepsilon}\right).
\]
Suppose by contradiction that there exists~$C>0$ such that, for all~$u\in W^{s,p}(\Omega)$, 
\[
\|u\|_{W^{\widetilde s,\widetilde p}(\Omega)}\le C \|u\|_{W^{s,p}(\Omega)}.
\]
We exploit Lemma~\ref{lemmino0Omega} with~$\gamma:= s-\frac{N}{p}$
and~$\beta:=-1$. Hence, if~$\widetilde p\ne +\infty$,
\[
\|v_\varepsilon\|_{W^{\widetilde s,\widetilde p}(\Omega)}\ge \varepsilon^{-\frac{N}{p}+\frac{N}{\widetilde p}+s-\widetilde s}
\left(\varepsilon^{\widetilde s \widetilde p}\,\|u\|_{L^{\widetilde p}(\Omega)}^{\widetilde p}+ [u]_{W^{\widetilde s,\widetilde p}(B_R)}^{\widetilde p} \right)^\frac{1}{\widetilde p}
\]
and
\[
\|v_\varepsilon\|_{W^{\widetilde s,\widetilde p}(\Omega)}\le C \|v_\varepsilon\|_{W^{s,p}(\Omega)}\le C\left(\varepsilon^{sp}\,\|u\|_{L^p(\Omega)}^p+[u]_{W^{s,p}(\R^N)}^p \right)^\frac{1}{p}.
\]
The last two displays show that
\[
\varepsilon^{-\frac{N}{p}+\frac{N}{\widetilde p}+s-\widetilde s}
\left(\varepsilon^{\widetilde s \widetilde p}\,\|u\|_{L^{\widetilde p}(\Omega)}^{\widetilde p}+ [u]_{W^{\widetilde s,\widetilde p}(B_R)}^{\widetilde p} \right)^\frac{1}{\widetilde p}\le C\left(\varepsilon^{sp}\,\|u\|_{L^p(\Omega)}^p+[u]_{W^{s,p}(\R^N)}^p \right)^\frac{1}{p}.
\]
Recalling~\eqref{esponentenegativo2} and taking the limit as~$\varepsilon\searrow 0$, we get the desired contradiction.

If instead~$\widetilde p = +\infty$,
\[
\begin{split}
&\varepsilon^{s -\frac{N}{p}}\|u\|_{L^\infty(\Omega)}+\varepsilon^{s-\widetilde s -\frac{N}{p}}\sup_{\substack{x,y\in B_R \\x\ne y}}\frac{|u(x)-u(y)|}{|x-y|^{\widetilde s}}\le \|v_\varepsilon\|_{C^{0, \widetilde s}(\Omega)}\\
&\qquad\le C \|v_\varepsilon\|_{W^{s,p}(\Omega)}\le C\left(\varepsilon^{sp}\,\|u\|_{L^p(\Omega)}^p+[u]_{W^{s,p}(\R^N)}^p \right)^\frac{1}{p}.
\end{split}
\]
Recalling \eqref{esponentenegativo3} and taking the limit as $\varepsilon\searrow 0$, we get the desired contradiction in this case as well.
\end{proof}

\subsection{A useful compact embedding}
We now show that, when~$p\in [1, +\infty)$ and~$\widetilde s\in(0, s)$, the embedding of~$W^{s,p}(\Omega)$ in~$W^{\widetilde{s},p}(\Omega)$ is compact.
This result will be the core of the proof of Theorem~\ref{teorema3sp=N} but it will be also crucial in the proof of the optimality statements for Theorems~\ref{teorema2}, \ref{teorema2sp=N} and~\ref{teorema2sp>N}.

\begin{lemma}\label{orizzontalecompatto}
Let~$s\in (0,1]$ and~$p\in [1, +\infty)$. Let~$\Omega\subset \R^N$ be a bounded Lipschitz domain. Let~$\widetilde{s}\in[0,s)$.

Then, the space~$W^{s,p}(\Omega)$ is compactly embedded in~$W^{\widetilde{s},p}(\Omega)$.
\end{lemma}

\begin{proof}
If~$\widetilde s=0$, then~$W^{s,p}(\Omega)$ is compactly embedded in~$L^{p}(\Omega)$, thanks to~\cite[Theorem~7.1]{MR2944369}.
This is the desired claim in this case. 

If~$\widetilde{s}\in(0,s)$, there exists~$\widetilde{\theta}\in (0,1)$ such that~$\widetilde{s}=(1-\widetilde{\theta})s$. Hence, we can apply~\cite[Theorem~1]{MR3813967} (recalled in the forthcoming Theorem~\ref{THMbrezismironuescu}) with~$s_1=0$, $s_2=s$ 
and~$p_1=p_2=p$ to see that there exists a constant~$C>0$ such that, for any~$u\in W^{s,p}(\Omega)$,
\begin{equation}\label{interpolazioneorizzontale}
\|u\|_{W^{\widetilde{s},p}(\Omega)}\leq C \|u\|_{L^p(\Omega)}^{\widetilde{\theta}}\|u\|_{W^{s,p}(\Omega)}^{1-\widetilde{\theta}}.
\end{equation}

Moreover, we know that the space~$W^{s,p}(\Omega)$ is compactly embedded in~$L^p(\Omega)$ (see e.g.~\cite[Theorem~7.1]{MR2944369}). Namely, for any sequence~$u_n$ in~$W^{s,p}(\Omega)$ which converges weakly to some~$u$ in~$W^{s,p}(\Omega)$, we have that~$u_n$ converges strongly to~$u$ in~$L^p(\Omega)$, as~$n\to+\infty$.

Therefore, from~\eqref{interpolazioneorizzontale} we infer that
\[
\|u_n-u\|_{W^{\widetilde{s},p}(\Omega)}\leq C \|u_n-u\|_{L^p(\Omega)}^{\widetilde{\theta}}\|u_n-u\|_{W^{s,p}(\Omega)}^{1-\widetilde{\theta}}.
\]
This implies that
\[
\lim_{n\to +\infty}\|u_n-u\|_{W^{\widetilde{s},p}(\Omega)}=0,
\]
which proves the desired compact embedding of~$W^{s,p}(\Omega)$ in~$W^{\widetilde{s},p}(\Omega)$.
\end{proof}

\subsection{Optimality of Theorems~\ref{teorema1} and~\ref{teorema2} and Corollary~\ref{corollariocompattezza}}

We are now in the position to consider the case~$sp<N$ and to show that Theorems~\ref{teorema1} and~\ref{teorema2} and Corollary~\ref{corollariocompattezza} are optimal.

\begin{proposition}\label{Thm1optimal}
Theorem~\ref{teorema1} is optimal, namely if~$\widetilde s$ and~$ \widetilde p$ do not satisfy~\eqref{insiemeimmersioneRN}, then the space $W^{s, p}(\R^N)$ is not continuously embedded in $W^{\widetilde s,  \widetilde p}(\R^N)$.
\end{proposition}

\begin{proof}
We notice that if~$\widetilde s$ and~$ \widetilde p$ do not satisfy~\eqref{insiemeimmersioneRN}, then one of the following holds:
\begin{equation*}
{\mbox{either }}\quad
\begin{cases}
0\le\widetilde{s}\le 1,\\
1\le\widetilde{p}<p,
\end{cases} \quad\mbox{ or }\quad\begin{cases}
s <\widetilde s \le 1,\\
p\le\widetilde p\le +\infty,
\end{cases}\quad\mbox{ or }\quad \begin{cases}
0\le\widetilde s \le s,\\
%%\dfrac{s p p^*_s}{sp + (p^*_s -p)\widetilde s}
\frac{Np}{N-(s-\widetilde s)p}<\widetilde p\le +\infty,
\end{cases}
\end{equation*}
namely either~\eqref{parentesilemma1} or~\eqref{parentesilemma2} holds true.

Now, if~\eqref{parentesilemma1} holds, then the desired result follows from Lemma~\ref{lemmino1}, while 
if~\eqref{parentesilemma2} is verified, then
Lemma~\ref{lemmino2} in the case~$\Omega=\R^N$ yields the desired result.
\end{proof}

\begin{proposition}\label{Thm2optimal}
Theorem~\ref{teorema2} is optimal, namely if~$\widetilde s$ and~$ \widetilde p$ do not satisfy~\eqref{insiemeimmersioneOmega}, then the space $W^{s, p}(\Omega)$ is not continuously embedded in $W^{\widetilde s,  \widetilde p}(\Omega)$.
\end{proposition}

\begin{proof}
We notice that, if~$\widetilde s$ and~$\widetilde p$ do not satisfy~\eqref{insiemeimmersioneOmega}, then one of the following holds:
\begin{equation}\label{trepossibilita2}
{\mbox{either }}\quad
\begin{cases}
s\le\widetilde{s}\le 1,\\
1\le\widetilde{p}<p,
\end{cases} \quad\mbox{ or }\quad\begin{cases}
s <\widetilde s \le 1,\\
p\le\widetilde p\le +\infty,
\end{cases}\quad\mbox{ or }\quad \begin{cases}
0\le\widetilde s \le s,\\
%%\dfrac{s p p^*_s}{sp + (p^*_s -p)\widetilde s}
\frac{Np}{N-(s-\widetilde s)p}<\widetilde p\le +\infty.
\end{cases}
\end{equation}
We notice that the second and third cases in~\eqref{trepossibilita2} coincide with~\eqref{parentesilemma2}. Hence, if this is the case, the desired result follows by Lemma~\ref{lemmino2}.

If not, the first case in~\eqref{trepossibilita2} occurs. We distinguish two cases. If~$\widetilde s = s$, then the desired result can be deduced by~\cite[Theorem~1.1]{MR3357858}.

If instead~$\widetilde s \in( s,1]$, then we assume by contradiction that~$W^{s, p}(\Omega)$ is continuously embedded in~$W^{\widetilde s, \widetilde p}(\Omega)$.
Moreover, by Lemma~\ref{orizzontalecompatto}, we have that~$W^{\widetilde s, \widetilde p}(\Omega)$ is compactly embedded 
in~$W^{s, \widetilde p}(\Omega)$.
Accordingly, we obtain that~$W^{s, p}(\Omega)$ is continuously embedded in~$W^{s, \widetilde p}(\Omega)$,
which is in contradiction with~\cite[Theorem~1.1]{MR3357858}.
\end{proof}

\begin{proposition}\label{Thm3optimal}
Corollary~\ref{corollariocompattezza} is optimal,
namely if~$\widetilde s$ and~$ \widetilde p$ do not satisfy~\eqref{insiemeimmersionecompatta2}, then the space $W^{s, p}(\Omega)$ is not compactly embedded in $W^{\widetilde s,  \widetilde p}(\Omega)$.
\end{proposition}

\begin{proof}
In light of Proposition~\ref{Thm2optimal}, if~$\widetilde s$ and~$ \widetilde p$ do not satisfy~\eqref{insiemeimmersioneOmega}, then~$W^{s, p}(\Omega)$ is not compactly embedded in~$W^{\widetilde s, \widetilde p}(\Omega)$.
Thus, it only remains to show that the points which belong to the critical curve are such that the compact embedding stated in Corollary~\ref{corollariocompattezza} does not hold.

To this purpose, let~$s\in (0,1]$ and~$p\in [1, +\infty)$ be such that~$sp<N$. Let~$\widetilde s$ and~$\widetilde p$ satisfy
\begin{equation}\label{stildeptilde}
0\le \widetilde s \le s\qquad\mbox{ and }\qquad 
\frac{Np}{N-sp}=\frac{N\widetilde p}{N-\widetilde s\widetilde p.}
%% p^*_s = \widetilde p^*_{\widetilde s}.
\end{equation}
Notice that the second condition in~\eqref{stildeptilde}
means that the fractional Sobolev exponents for the couples~$(s,p)$
and~$(\widetilde s,\widetilde p)$ coincide.

Up to translation, suppose that~$0\in\Omega$ and let~$R>0$ be such that~$B_R\subset\Omega$. Let~$u\in C^\infty_0(B_R)\setminus\{0\}$ and, for~$\varepsilon\in(0,1)$, let
\[
v_{\varepsilon}(x):= \varepsilon^{-\frac{N-sp}{p}} u\left(\frac{x}{\varepsilon}\right).
\]

We remark that, for all~$x\in B_R\setminus\{0\}$,
there exists~$\varepsilon_0\in(0,1)$, depending on~$x$,
such that if~$\varepsilon\in(0,\varepsilon_0]$
we have that~$|x|/\varepsilon>R$, and therefore~$v_{\varepsilon}(x)=0$.

This says that 
\begin{equation}\label{vepstozero}
{\mbox{$v_\varepsilon\to 0$ a.e. in~$\Omega$ as~$\varepsilon\searrow0$.}}
\end{equation}

Furthermore,
formulas~\eqref{agg00} and~\eqref{zse4rdfrtgyujhio1}
(used here with~$\gamma:=-(N-sp)/p$)
give that
\[
\|v_{\varepsilon}\|^p_{L^p(\Omega)} = \varepsilon^{sp} \|u\|^p_{L^p(\Omega)}\qquad
{\mbox{and}}\qquad
[v_{\varepsilon}]^p_{W^{s, p}(\Omega)} \le [u]^p_{W^{s, p}(\R^N)},
\]
which in turn imply that~$v_{\varepsilon}$ is bounded in~$W^{s, p}(\Omega)$. 

Thus, if the embedding of~$W^{s, p}(\Omega)$ into~$W^{\widetilde s, \widetilde p}(\Omega)$ were compact, up to a subsequence, $v_\varepsilon$ would converge to some~$v$ in~$W^{\widetilde s, \widetilde p}(\Omega)$, and therefore also pointwise in~$\Omega$.
In light of~\eqref{vepstozero}, we would then have that
\begin{equation}\label{4567328zxcvbnlkjhgfdsaqwertyuio098765432}
{\mbox{$v_\varepsilon$ converges to~$0$ in~$W^{\widetilde s, \widetilde p}(\Omega)$ as~$\varepsilon\searrow0$.}}
\end{equation}

On the other hand, formula~\eqref{agg00}, together
with the second assumption in~\eqref{stildeptilde}, gives that
\begin{equation*}
\|v_{\varepsilon}\|^{\widetilde p}_{L^{\widetilde p}(\Omega)}  = \varepsilon^{\widetilde p \widetilde s} \|u\|^{\widetilde p}_{L^{\widetilde p}(\Omega)}\qquad\mbox{ and }\qquad 
[v_{\varepsilon}]^{\widetilde p}_{W^{\widetilde s, \widetilde p}(\Omega)}
\ge [u]^{\widetilde p}_{W^{\widetilde s, \widetilde p}(B_R)} 
.
\end{equation*}
Hence,
\[
\lim_{\varepsilon\to 0}\left(\|v_{\varepsilon}\|^{\widetilde p}_{L^{\widetilde p}(\Omega)} + [v_{\varepsilon}]^{\widetilde p}_{W^{\widetilde s, \widetilde p}(\Omega)} \right)\ge
[u]^{\widetilde p}_{W^{\widetilde s, \widetilde p}(B_R)}>0.
\]
This is in contradiction with~\eqref{4567328zxcvbnlkjhgfdsaqwertyuio098765432},
and therefore the desired claim is established.
\end{proof}

\subsection{Optimality of Theorems~\ref{teorema1sp=N}, \ref{teorema2sp=N} and~\ref{teorema3sp=N}}

We now address the case~$sp=N$ and we show that Theorems~\ref{teorema1sp=N}, \ref{teorema2sp=N} and~\ref{teorema3sp=N} are optimal. 

\begin{proposition}\label{Thm5optimal}
Theorem~\ref{teorema1sp=N} is optimal, namely if~$\widetilde s$ and~$ \widetilde p$ do not satisfy~\eqref{insiemeimmersioneRNsp=N}, then the space $W^{s, p}(\R^N)$ is not continuously embedded in $W^{\widetilde s,  \widetilde p}(\R^N)$.
\end{proposition}

\begin{proof}
Assume that~$\widetilde s$ and~$ \widetilde p$ do not satisfy~\eqref{insiemeimmersioneRNsp=N}. Then one of the following\footnote{With a slight abuse of notation, when~$\widetilde s=0$ in
the second possibility in~\eqref{trepossibilita5} and~\eqref{trepossibilita6}, we allow~$\widetilde p=+\infty$.} holds:
\begin{equation}\label{trepossibilita5}
{\mbox{either }}\quad
\begin{cases}
0\le \widetilde{s}\le 1,\\
1\le\widetilde{p}<p,
\end{cases} \quad\mbox{ or }\quad\begin{cases}
s <\widetilde s \le 1,\\
p\le\widetilde p\le +\infty,
\end{cases}\quad\mbox{ or }\quad \begin{cases}
0\le\widetilde s \le s,\\
\dfrac{N}{\widetilde s}<\widetilde p\le +\infty.
\end{cases}
\end{equation}
The first possibility in~\eqref{trepossibilita5} coincides
with~\eqref{parentesilemma1}, and therefore
the desired result follows from Lemma~\ref{lemmino1}.

The second and third possibilities in~\eqref{trepossibilita5}
give~\eqref{parentesilemma3} if~$N\ge2$, and so
Lemma~\ref{lemmino3} in the case~$\Omega=\R^N$ with~$N\ge2$ yields the desired result.

Also, we point out that the second and third possibilities in~\eqref{trepossibilita5} when~$\widetilde s\neq 0$
give~\eqref{parentesilemma3BIS} if~$N=1$. Thus, the desired result
when~$\Omega=\R$ follows from Lemma~\ref{lemmino3}.

In the case~$N=1$, $\widetilde s=0$ and~$\widetilde p=+\infty$,
we remark that one of the assumptions in Theorem~\ref{teorema1sp=N} is that~$s\neq p$, and therefore, in this case, we have that~$p\in(1,+\infty)$. This allows us
to exploit~\cite[formulas~(1.9)--(1.10) in
Theorem~B]{MR3990737} and obtain the desired conclusion.
\end{proof}

\begin{proposition}\label{Thm6optimal}
Theorem~\ref{teorema2sp=N} is optimal, namely if~$\widetilde s$ and~$ \widetilde p$ do not satisfy~\eqref{insiemeimmersioneOmegasp=N}, then the space $W^{s, p}(\Omega)$ is not continuously embedded in $W^{\widetilde s,  \widetilde p}(\Omega)$.
\end{proposition}

\begin{proof}
Suppose that~$\widetilde s$ and~$ \widetilde p$ do not satisfy~\eqref{insiemeimmersioneOmegasp=N}. 
Then one of the following holds:
\begin{equation}\label{trepossibilita6}
{\mbox{either }}\quad
\begin{cases}
s\le\widetilde{s}\le 1,\\
1\le\widetilde{p}<p,
\end{cases} \quad\mbox{ or }\quad\begin{cases}
s <\widetilde s \le 1,\\
p\le\widetilde p\le +\infty,
\end{cases}\quad\mbox{ or }\quad \begin{cases}
0\le\widetilde s \le s,\\
\dfrac{N}{\widetilde s}<\widetilde p\le +\infty,
\end{cases}
\end{equation}
If the first case in~\eqref{trepossibilita6} occurs, we
distinguish two cases. If~$\widetilde s = s$, then the desired result can be deduced by~\cite[Theorem~1.1]{MR3357858}.

If instead~$\widetilde s\in(s,1]$, we assume by contradiction that~$
W^{s, p}(\Omega)$ is continuously embedded in~$W^{\widetilde s, \widetilde p}(\Omega)$.
Also, by Lemma~\ref{orizzontalecompatto}, we have that~$W^{\widetilde s, \widetilde p}(\Omega)$ is compactly embedded in~$W^{s, \widetilde p}(\Omega)$.
As a result, we obtain that~$W^{s, p}(\Omega)$ is continuously embedded in $W^{s, \widetilde p}(\Omega)$,
which is in contradiction with~\cite[Theorem~1.1]{MR3357858}.

The second and third possibilities in~\eqref{trepossibilita6}
give either~\eqref{parentesilemma3} if~$N\ge2$
or~\eqref{parentesilemma3BIS} if~$N=1$ and~$\widetilde s\neq 0$, and therefore
Lemma~\ref{lemmino3} in these cases yields the desired result.

It remains to check the case~$N=1$ and~$\widetilde s=0$.
In this situation we have that~$\widetilde p=+\infty$. Also,
we observe that~$s\neq p$ by the assumptions
in the statement of Theorem~\ref{teorema2sp=N}.
Thus, in this case, we have that~$p\in(1,+\infty)$ and we can
exploit~\cite[formulas~(1.9)--(1.10) in
Theorem~B]{MR3990737} to finish the proof.
\end{proof}

\begin{proposition}\label{Thm7optimal}
Theorem~\ref{teorema3sp=N} is optimal, namely if~$\widetilde s$ and~$ \widetilde p$ do not satisfy~\eqref{2insiemeimmersioneOmegasp=N}, then the space $W^{s, p}(\Omega)$ is not compactly embedded in $W^{\widetilde s,  \widetilde p}(\Omega)$.
\end{proposition}

\begin{proof}
If~$\widetilde s$ and~$ \widetilde p$ do not satisfy~\eqref{insiemeimmersioneOmegasp=N}, then~$W^{s, p}(\Omega)$ is not compactly embedded in~$W^{\widetilde s, \widetilde p}(\Omega)$, thanks to Proposition~\ref{Thm6optimal}.
Hence, it only remains to show that the points which belong to the critical curve~$\widetilde s \widetilde p =N$ are such that the compact embedding in Theorem~\ref{Thm7optimal} does not hold.

To this purpose, let~$s\in (0,1]$ and~$p\in[1,+\infty)$ be such that~$sp=N$ and~$s\neq p$. Let~$\widetilde s$ and~$\widetilde p$ satisfy
\begin{equation}\label{2stildeptilde}
\widetilde s \in(0, s]\qquad {\mbox{and}}\qquad \widetilde s \widetilde p =N.
\end{equation}
Up to a translation, we suppose that~$0\in\Omega$ and we let~$R>0$
such that~$B_R\subset\Omega$. Let~$u\in C^\infty_0(B_R)\setminus\{0\}$
and, for any~$\varepsilon\in(0,1)$, let
\[
v_{\varepsilon}(x):= u\left(\frac{x}{\varepsilon}\right).
\]

For all~$x\in B_R\setminus\{0\}$ there exists~$\varepsilon_0\in(0,1)$ such that, for all~$\varepsilon\in(0,\varepsilon_0]$, we have that~$|x|/\varepsilon>R$ and thus, by construction, we have that~$v_{\varepsilon}(x)=0$.
As a result,
\begin{equation}\label{2vepstozero}
{\mbox{$v_\varepsilon\to 0$ a.e. in~$\Omega$ as~$\varepsilon\searrow0$.}}\end{equation}

In addition, we exploit formulas~\eqref{agg00} and~\eqref{zse4rdfrtgyujhio1} (used here with~$\gamma:=0$) to see that
\[
\|v_{\varepsilon}\|^p_{L^p(\Omega)} = \varepsilon^{sp} \|u\|^p_{L^p(\Omega)}
\qquad {\mbox{and}} \qquad 
[v_{\varepsilon}]^p_{W^{s, p}(\Omega)} \le  [u]^p_{W^{s, p}(\R^N)}.
\]
This gives that~$v_\varepsilon$ is a bounded sequence in~$W^{s, p}(\Omega)$.

Accordingly, if the embedding of~$W^{s, p}(\Omega)$ into~$W^{\widetilde s, \widetilde p}(\Omega)$ were compact, up to
a subsequence, we would have that~$v_\varepsilon$ converges
to some~$v$ in~$W^{\widetilde s, \widetilde p}(\Omega)$ as~$\varepsilon\searrow0$. This and~\eqref{2vepstozero} imply that
\begin{equation}\label{mnbvcxzasdfghjkoiuytre23456789mnbvcx}
{\mbox{$v_\varepsilon$ converges to~$0$ in~$W^{\widetilde s, \widetilde p}(\Omega)$ as~$\varepsilon\searrow0$.}}
\end{equation}

Furthermore, formula~\eqref{agg00} and
the second assumption in~\eqref{2stildeptilde} give that
\begin{equation*}
\|v_{\varepsilon}\|^{\widetilde p}_{L^{\widetilde p}(\Omega)}  = \varepsilon^{\widetilde p \widetilde s} \|u\|^{\widetilde p}_{L^{\widetilde p}(\Omega)}\qquad\mbox{ and }\qquad 
[v_{\varepsilon}]^{\widetilde p}_{W^{\widetilde s, \widetilde p}(\Omega)}
\ge [u]^{\widetilde p}_{W^{\widetilde s, \widetilde p}(B_R)}.
\end{equation*}
We thereby obtain that
\[
\lim_{\varepsilon\to 0}\left(\|v_{\varepsilon}\|^{\widetilde p}_{L^{\widetilde p}(\Omega)} + [v_{\varepsilon}]^{\widetilde p}_{W^{\widetilde s, \widetilde p}(\Omega)} \right)\ge
[u]^{\widetilde p}_{W^{\widetilde s, \widetilde p}(B_R)}>0,
\]
which is in contradiction with~\eqref{mnbvcxzasdfghjkoiuytre23456789mnbvcx},
and therefore the desired claim is established.
\end{proof}

\subsection{Optimality of Theorems~\ref{teorema1sp>N}, \ref{teorema2sp>N} and~\ref{teorema3sp>N}}

Finally, we discuss the optimality of the results stated in the case~$sp>N$.

\begin{proposition}\label{Thm8optimal}
Theorem~\ref{teorema1sp>N} is optimal, namely if~$\widetilde s$ and~$ \widetilde p$ do not satisfy~\eqref{insiemeimmersioneRNsp>N}, then the space $W^{s, p}(\R^N)$ is not continuously embedded in $W^{\widetilde s,  \widetilde p}(\R^N)$.
\end{proposition}

\begin{proof}
Observe that, if~$\widetilde s$ and~$\widetilde p$ do not satisfy~\eqref{insiemeimmersioneRNsp>N}, then one of the following holds:
\begin{equation}\label{trepossibilita8}
{\mbox{either}}\quad
\begin{cases}
0\le\widetilde{s}\le 1,\\
1\le\widetilde{p}<p,
\end{cases} \quad\mbox{or}\quad\begin{cases}
s <\widetilde s \le 1,\\
p\le\widetilde p\le +\infty,
\end{cases}\quad\mbox{or}\quad \begin{cases}
\dfrac{sp-N}{p}<\widetilde s \le s,\\
\dfrac{Np}{N-(s-\widetilde s)p}<\widetilde p\le +\infty,
\end{cases}
\end{equation}
namely either~\eqref{parentesilemma1} or~\eqref{parentesilemma4} holds true.

Therefore, the desired result follows from Lemma~\ref{lemmino1} if~\eqref{parentesilemma1} holds true and from
Lemma~\ref{lemmino4} in the case~$\Omega=\R^N$ 
if~\eqref{parentesilemma4} is satisfied.
\end{proof}

\begin{proposition}\label{Thm9optimal}
Theorem~\ref{teorema2sp>N} is optimal, namely if~$\widetilde s$ and~$ \widetilde p$ do not satisfy~\eqref{insiemeimmersioneOmegasp>N}, then the space $W^{s, p}(\Omega)$ is not continuously embedded in $W^{\widetilde s,  \widetilde p}(\Omega)$.
\end{proposition}

\begin{proof}
Notice that, if~$\widetilde s$ and~$ \widetilde p$ do not satisfy~\eqref{insiemeimmersioneOmegasp>N}, then one of the following holds:
\begin{equation}\label{trepossibilita9}
{\mbox{either}}\quad
\begin{cases}
s\le\widetilde{s}\le 1,\\
1\le\widetilde{p}<p,
\end{cases} \quad\mbox{or}\quad\begin{cases}
s <\widetilde s \le 1,\\
p\le\widetilde p\le +\infty,
\end{cases}\quad\mbox{or}\quad \begin{cases}
\dfrac{sp-N}{p}<\widetilde s \le s,\\
\dfrac{Np}{N-(s-\widetilde s)p}<\widetilde p\le +\infty.
\end{cases}
\end{equation}

The second and third cases in~\eqref{trepossibilita9} coincide with~\eqref{parentesilemma4}. Hence, if this is the case, the desired result follows from Lemma~\ref{lemmino4}.

If not, the first case in~\eqref{trepossibilita9} occurs.
We now distinguish two cases. If~$\widetilde s = s$, then the desired result can be deduced by~\cite[Theorem~1.1]{MR3357858}.

If instead~$\widetilde s\in(s,1]$, we assume by contradiction that~$W^{s, p}(\Omega)$ is continuously embedded in~$W^{\widetilde s, \widetilde p}(\Omega)$.
Also, by Lemma~\ref{orizzontalecompatto}, we have that~$W^{\widetilde s, \widetilde p}(\Omega)$ is continuously embedded in~$W^{s, \widetilde p}(\Omega)$.
Thus, combining these two facts, we obtain that~$W^{s, p}(\Omega)$ is continuously embedded in~$W^{s, \widetilde p}(\Omega)$,
which is in contradiction with~\cite[Theorem~1.1]{MR3357858}.
\end{proof}

\begin{proposition}\label{Thm10optimal}
Theorem~\ref{teorema3sp>N} is optimal, namely if~$\widetilde s$ and~$ \widetilde p$ do not satisfy~\eqref{2insiemeimmersioneOmegasp>N}, then the space $W^{s, p}(\Omega)$ is not compactly embedded in $W^{\widetilde s,  \widetilde p}(\Omega)$.
\end{proposition}

\begin{proof}
In light of Proposition~\ref{Thm9optimal}, if~$\widetilde s$ and~$\widetilde p$ do not satisfy~\eqref{insiemeimmersioneOmegasp>N}, then~$W^{s, p}(\Omega)$ is not compactly embedded in~$W^{\widetilde s, \widetilde p}(\Omega)$.
Thus, it only remains to show that the points which belong to the curve
\[
\widetilde p=\frac{Np}{N-p(s-\widetilde s)}
\]
are such that the compact embedding stated in Theorem~\ref{teorema3sp>N} does not hold.

To this purpose, let~$s\in (0,1]$ and~$p\geq 1$ be such that~$sp>N$. Let~$\widetilde s$ and~$\widetilde p$ satisfy
\begin{equation}\label{3stildeptilde}
\widetilde s\in[0, s)\qquad{\mbox{and}}\qquad 
\widetilde s -\frac{N}{\widetilde p}=s-\frac{N}{p}.
\end{equation}

Without loss of generality, we suppose that~$0\in\Omega$ and we let~$R>0$ such that~$B_R\subset\Omega$.
Let also~$u\in C^\infty_0(B_R)\setminus\{0\}$ and,
for any~$\varepsilon\in(0,1)$,
\[
v_{\varepsilon}(x):= \varepsilon^{s-\frac{N}{p}} u\left(\frac{x}{\varepsilon}\right).
\]

We observe that for all~$x\in B_R\setminus\{0\}$ there exists~$\varepsilon_0\in(0,1)$ such that, for all~$\varepsilon\in(0,\varepsilon_0]$,
we have that~$|x|/\varepsilon>R$, and as a result~$v_\varepsilon(x)=$. This gives that 
\begin{equation}\label{2vepstozero333}
{\mbox{$v_\varepsilon\to 0$ a.e. in~$\Omega$ as~$\varepsilon\searrow0$.}}\end{equation}

Now, we use formulas~\eqref{agg00} and~\eqref{zse4rdfrtgyujhio1} (with~$\gamma:=s-N/p$) to see that
\[
\|v_{\varepsilon}\|^p_{L^p(\Omega)} = \varepsilon^{sp} \|u\|^p_{L^p(\Omega)}
\qquad {\mbox{and}} \qquad 
[v_{\varepsilon}]^p_{W^{s, p}(\Omega)} \le  [u]^p_{W^{s, p}(\R^N)}.
\]
As a consequence, $v_\varepsilon$ is a bounded sequence in~$W^{s, p}(\Omega)$.

Thus, if the embedding of~$W^{s,p}(\Omega)$
into~$W^{\widetilde s,\widetilde p}(\Omega)$ were compact,
we would have that~$v_\varepsilon$ converges, up to a subsequence,
to some function~$v$ in~$W^{\widetilde s,\widetilde p}(\Omega)$
as~$\varepsilon\searrow0$. {F}rom this and~\eqref{2vepstozero333} we would then have that
\begin{equation}\label{mnbvc98765432qwertyu3456789mnbvcx}
{\mbox{$v_\varepsilon$ converges to~$0$ in~$W^{\widetilde s, \widetilde p}(\Omega)$ as~$\varepsilon\searrow0$.}}
\end{equation}

Also, formula~\eqref{agg00} and
the second assumption in~\eqref{3stildeptilde} give that
\begin{equation*}
\|v_{\varepsilon}\|^{\widetilde p}_{L^{\widetilde p}(\Omega)}  = \varepsilon^{\widetilde p \widetilde s} \|u\|^{\widetilde p}_{L^{\widetilde p}(\Omega)}\qquad\mbox{ and }\qquad 
[v_{\varepsilon}]^{\widetilde p}_{W^{\widetilde s, \widetilde p}(\Omega)}
\ge [u]^{\widetilde p}_{W^{\widetilde s, \widetilde p}(B_R)}.
\end{equation*}
Consequently,
\[
\lim_{\varepsilon\to 0}\left(\|v_{\varepsilon}\|^{\widetilde p}_{L^{\widetilde p}(\Omega)} + [v_{\varepsilon}]^{\widetilde p}_{W^{\widetilde s, \widetilde p}(\Omega)} \right)\ge
[u]^{\widetilde p}_{W^{\widetilde s, \widetilde p}(B_R)}>0,
\]
which is in contradiction with~\eqref{mnbvc98765432qwertyu3456789mnbvcx},
thus proving the desired claim.
\end{proof}

\section{Auxiliary results}\label{sectionauxiliary}
In this section we provide some preliminary results that will be used to prove our main theorems.

\subsection{An interpolation result}
We start by recalling the following interpolation result due to Brezis and Mironescu (see~\cite[Theorem~1]{MR3813967}). We state it in a way which is suitable for our analysis:

\begin{theorem}\label{THMbrezismironuescu}
Let~$0\leq s_1\leq s_2\leq 1$ and~$p_1$, $p_2\in [1,+\infty]$
with~$s_2\ne p_2$.
Let~$\Omega$ be either~$\R^N$ or a bounded Lipschitz domain of~$\R^N$.

For every~$\theta\in [0,1]$, let 
\[
s_\theta:=\theta s_1 +(1-\theta)s_2 \qquad \mbox{and} \qquad
p_\theta:=\frac{p_1 \,p_2}{(1-\theta)p_1+\theta p_2}.
\]

Then, there exists a positive constant~$C=C(s_1,s_2,p_1,p_2,\theta, \Omega)$ such that, for any~$u\in W^{s_1,p_1}(\Omega)\cap W^{s_2,p_2}(\Omega)$,
\[
\|u\|_{W^{s_\theta,p_\theta}(\Omega)}
\leq C \|u\|_{W^{s_1,p_1}(\Omega)}^\theta
\|u\|_{W^{s_2,p_2}(\Omega)}^{1-\theta} .
\]
\end{theorem}

\subsection{A suitable curve of points}

Now, we show that any points~$\widetilde s$ and~$\widetilde p$ satisfying~\eqref{insiemeimmersioneRN} belong to a suitable curve.
For this, if~$sp<N$, we define the fractional Sobolev exponent
$$ p^*_s:=\frac{Np}{N-sp}.$$

\begin{lemma}\label{curvaq11}
Let~$s\in (0,1]$ and~$p\in [1, +\infty)$ be such that~$sp<N$. Assume that~$\widetilde{s}$ and~$\widetilde{p}$ satisfy~\eqref{insiemeimmersioneRN}. 

Then, there exists~$q\in [p,p^*_s]$ such that the curve~$\gamma_q:[0,1]\to \R^2$ defined as
\begin{equation}\label{curvalemma}
\gamma_q(\theta):=\left((1-\theta)s,\frac{pq}{q-\theta(q-p)} \right)
\end{equation}
contains the point~$(\widetilde{s},\widetilde{p})$, namely there exists~$\widetilde{\theta}\in [0,1]$ such that~$\gamma_q({\widetilde{\theta}})=(\widetilde{s},\widetilde{p})$.
\end{lemma}

\begin{proof}
Notice that if~$\widetilde s=s$, from~\eqref{insiemeimmersioneRN}
we see that~$\widetilde p=p$. In this case we then have
that~$\gamma_q(0)=(\widetilde{s},\widetilde{p})$
for every~$q\in [p,p^*_s]$. Thus, the couple~$(\widetilde s,\widetilde p)$
belongs to the curve~$\gamma_q$. 

Suppose instead that~$\widetilde s\in[0,s)$ and~$\widetilde{p}=p$. In this case, taking~$q:=p$, we see that~$\gamma_q(\theta)=((1-\theta)s,p)$.
Therefore, there exists~$\widetilde\theta\in(0,1]$ such that~$\gamma_q(\widetilde\theta)=(\widetilde s,\widetilde p)$, as desired.

Given these preliminary observations, from now on we suppose that~$\widetilde s\neq s$ and~$\widetilde p\neq p$.

Moreover, we remark that the condition~$\widetilde s\widetilde p=sp$ is incompatible with~\eqref{insiemeimmersioneRN},
since
\[
\widetilde{s}\widetilde{p}
\leq 
\frac{Np \widetilde s}{N-(s-\widetilde s)p} 
=
\frac{\widetilde{s}spp^*_s}{sp+(p^*_s-p)\widetilde{s}}
=\left( \frac{\widetilde{s}p^*_s}{\widetilde{s}p^*_s+p(s-\widetilde{s})}\right)sp<sp.
\]
Therefore, from now on,
we also assume that~$\widetilde s\widetilde p\neq sp$.

Now, we observe that the couple~$(\widetilde{s},\widetilde{p})$ belongs to the curve~$\gamma_q$ if and only if it is contained in the graph of the function~$f_q:[0,s]\to \R$ defined as
\[
f_q(t):=\dfrac{spq}{sp+(q-p)t}  .
\]

We set
\[
q:=\frac{p\widetilde{p}(s-\widetilde{s})}{sp-\widetilde{s}\widetilde{p}}
\]
and we have that
\[
\widetilde{p}=\dfrac{spq}{sp+(q-p)\widetilde{s}}=f_q(\widetilde{s}).
\]

Hence, it remains to check that
\begin{equation}\label{1qaz2wsx3edc4rfv5tgb}
q \in[p,p^*_s].\end{equation} 
For this, we notice that, in light of~\eqref{insiemeimmersioneRN},
\[
q> \frac{p^2(s-\widetilde{s})}{p(s-\widetilde{s})}=p
\]
and
\[
q\leq \dfrac{\dfrac{sp^2 p^*_s\, (s-\widetilde s)}{sp+(p^*_s-p)\widetilde{s}}}{sp-\dfrac{\widetilde{s}spp^*_s}{sp+(p^*_s-p)\widetilde{s}}}=\frac{p^*_s (s-\widetilde{s})}{1-\widetilde{s}}\le p^*_s.
\]
These considerations proves~\eqref{1qaz2wsx3edc4rfv5tgb}, as desired.
\end{proof}

We now check that if~$\widetilde s$ and~$\widetilde p$ satisfy~\eqref{insiemeimmersionecompatta}, they belong to the curve~$\gamma_q$, as stated in this result:

\begin{lemma}\label{curvaq13}
Let~$s\in (0,1]$ and~$p\in [1, +\infty)$ be such that~$sp<N$. Let~$\widetilde{s}$ and~$\widetilde{p}$ satisfy~\eqref{insiemeimmersionecompatta}. 

Then, there exists~$q\in [1,p^*_s)$ such that the curve~$\gamma_q:[0,1]\to \R^2$ defined in~\eqref{curvalemma} contains the point~$(\widetilde{s},\widetilde{p})$, namely there exists~$\widetilde{\theta}\in (0,1]$ such that~$\gamma_q({\widetilde{\theta}})=(\widetilde{s},\widetilde{p})$.
\end{lemma}

\begin{proof}
The proof of this result is similar to the proof of Lemma~\ref{curvaq11}. For the reader convenience, we provide all the details here.

If~$\widetilde s\in[0,s)$ and~$\widetilde{p}=p$,
we take~$q:=p$ and we see that~$\gamma_q(\theta)=((1-\theta)s,p)$.
Therefore, there exists~$\widetilde\theta\in[0,1]$ such that~$\gamma_q(\widetilde\theta)=(\widetilde s,\widetilde p)$, as desired.
Hence, we will now suppose that~$\widetilde p\neq p$.

{F}rom~\eqref{insiemeimmersionecompatta} we also have that
\[
\widetilde{s}\widetilde{p}\le \frac{Np \widetilde s}{N-(s-\widetilde s)p}
=\frac{\widetilde{s}spp^*_s}{sp+(p^*_s-p)\widetilde{s}}
=\left( \frac{\widetilde{s}p^*_s}{\widetilde{s}p^*_s+p(s-\widetilde{s})}\right)sp<sp.
\]

Now, we observe that the point~$(\widetilde{s},\widetilde{p})$ 
belongs to the curve~$\gamma_q$ in~\eqref{curvalemma}
if and only if it is contained in the graph of the function~$f_q:[0,s]\to \R$ defined as
\[
f_q(t):=
\dfrac{spq}{sp+(q-p)t} .
\]
We set
\begin{equation*}
q:=\frac{p\widetilde{p}(s-\widetilde{s})}{sp-\widetilde{s}\widetilde{p}}
\end{equation*}
and we point out that
\[
\widetilde{p}=
\dfrac{spq}{sp+(q-p)\widetilde{s}}=f_q(\widetilde{s}).
\]

Thus, the proof of Lemma~\ref{curvaq13} is complete if we show that~$q\in[1,p^*_s)$. 
For this, one notices that, by~\eqref{insiemeimmersionecompatta},
\[
q\geq \dfrac{\dfrac{sp^2 (s-\widetilde{s})}{sp-(p-1)\widetilde s}}{sp-\dfrac{\widetilde{s}sp}{sp-(p-1)\widetilde s}}=1
\]
and \begin{equation*}
q< \dfrac{\dfrac{sp^2 p^*_s (s-\widetilde{s})}{sp+(p^*_s-p)\widetilde{s}}}{sp-\dfrac{\widetilde{s}spp^*_s}{sp+(p^*_s-p)\widetilde{s}}}=\frac{p^*_s (s-\widetilde{s})}{1-\widetilde{s}}\le p^*_s.
\qedhere
\end{equation*}
\end{proof}

We now show that, under that assumption that the points~$\widetilde{s}$ and~$\widetilde{p}$ satisfy~\eqref{insiemeimmersioneRNsp=N}, a result similar to the one presented in Lemma~\ref{curvaq13} holds true when~$sp=N$.  We point out that this choice allows the point~$q$ to be selected in a larger interval.

\begin{lemma}\label{curvaq1sp=N}
Let~$s\in (0,1]$ and~$p\in [1, +\infty)$ be such that~$sp=N$.
Let~$\widetilde{s}$ and~$\widetilde{p}$ satisfy~\eqref{insiemeimmersioneRNsp=N}. 

Then, there exists~$q\in [p,+\infty)$ such that the curve~$\gamma_q:[0,1]\to \R^2$ defined in~\eqref{curvalemma} contains the point~$(\widetilde{s},\widetilde{p})$, namely there exists~$\widetilde{\theta}\in (0,1]$ such that~$\gamma_q({\widetilde{\theta}})=(\widetilde{s},\widetilde{p})$.
\end{lemma}

\begin{proof}
The proof of this result is similar to the proof of Lemmata~\ref{curvaq11}
and~\ref{curvaq13}. For the reader convenience, we provide all the details here.

If~$\widetilde s=s$, then~\eqref{insiemeimmersioneRNsp=N}
and the fact that~$sp=N$ give that~$\widetilde p=p$,
and therefore~$\gamma_q(0)=(\widetilde{s},\widetilde{p})$
for every~$q\in [p,+\infty)$.

Moreover, if~$\widetilde s\in(0,s)$ and~$\widetilde{p}=p$, we take~$q:=p$ and we see that~$\gamma_q(\theta)=((1-\theta)s,p)$.
Therefore, there exists~$\widetilde\theta\in[0,1]$ such that~$\gamma_q(\widetilde\theta)=(\widetilde s,\widetilde p)$, as desired.

Given these preliminary observations, from now on we suppose that~$\widetilde s\neq s$ and~$\widetilde p\neq p$.

Also, we note that if~$\widetilde{s}$ and~$\widetilde{p}$ satisfy condition~\eqref{insiemeimmersioneRNsp=N}, then~$
\widetilde{s}\widetilde{p}<N=sp$.

Now, we notice that the point~$(\widetilde{s},\widetilde{p})$ 
belongs to the curve~$\gamma_q$ in~\eqref{curvalemma}
if and only if it is contained in the graph of the function~$f_q:[0,s]\to \R$ defined as
\[
f_q(t):=\dfrac{spq}{sp+(q-p)t}.
\]
Thus, we set
\[
q:=\frac{p\widetilde{p}(s-\widetilde{s})}{sp-\widetilde{s}\widetilde{p}}
\]
and we find that
\[
\widetilde{p}=
\dfrac{spq}{sp+(q-p)\widetilde{s}}=f_q(\widetilde{s}).
\]
Since~$\widetilde{p}\geq p$, we have that
\[
q\geq \frac{p^2(s-\widetilde{s})}{p(s-\widetilde{s})}=p,
\] 
which concludes the proof.
\end{proof}

\subsection{An embedding result}
We now state a preliminary embedding result which will be useful to prove Theorem~\ref{teorema2}. We point out that our result is similar to~\cite[Lemma~2.1]{MR3910033}, but we provide here a self-contained proof that does not involve the use of Besov or other interpolation spaces.

\begin{lemma}\label{immersionidisotto}
Let~$s\in (0,1)$ and~$p\in [1, +\infty]$. Let~$\Omega\subset \R^N$ be a bounded Lipschitz domain. Let~$\widetilde{s}\in(0,s)$ and~$\widetilde{p}\in[1,p]$.

Then, there exists a positive constant~$C=C(N,s,p,\widetilde{s},\widetilde{p},\Omega)$ such that, for any~$u\in W^{s,p}(\Omega)$,
\[
\|u\|_{W^{\widetilde{s},\widetilde{p}}(\Omega)}
\leq C\|u\|_{W^{s,p}(\Omega)},
\]
namely, the space~$W^{s,p}(\Omega)$ is continuously embedded in~$W^{\widetilde{s},\widetilde{p}}(\Omega)$.
\end{lemma}

\begin{proof}
We first suppose that~$p\in[1,+\infty)$ and we claim that
\begin{equation}\label{lemmaimmersione1}
\|u\|_{L^{\widetilde{p}}(\Omega)}\le C \, \|u\|_{L^p(\Omega)}.
\end{equation}
Indeed, if~$\widetilde p= p$, then the claim is obvious. If instead~$\widetilde p\in [1,p)$, by the H\"older inequality with exponents~$p/\widetilde p$ and~$p/(p-\widetilde p)$,
\begin{equation*}
\|u\|_{L^{\widetilde{p}}(\Omega)}\le |\Omega|^{\frac{1}{\widetilde p}-\frac{1}{p}} \, \|u\|_{L^p(\Omega)},
\end{equation*} which gives~\eqref{lemmaimmersione1}.

We also show that
\begin{equation}\label{lemmaimmersione3}
[u]_{W^{\widetilde{s},\widetilde{p}}(\Omega)}\le C \, [u]_{W^{s,p}(\Omega)}.
\end{equation}
To this end, it is convenient to write
\[
\frac{|u(x)-u(y)|^{\widetilde{p}}}{|x-y|^{N+\widetilde{s}\widetilde{p}}}=\frac{|u(x)-u(y)|^{\widetilde{p}}}{|x-y|^{\frac{N\widetilde{p}}{p}+s\widetilde{p}}} \ 
\frac{1}{|x-y|^{\frac{N(p-\widetilde{p})}{p}-\widetilde{p}(s-\widetilde{s})}}.
\]
If~$\widetilde p=p$, then the claim in~\eqref{lemmaimmersione3}
follows from this and the fact that~$\Omega$ is bounded.

If instead~$\widetilde p\in[1,p)$,
we use the H\"older inequality with exponents~$p/\widetilde p$ and~$p/(p-\widetilde p)$ to find that
\begin{equation*}
\begin{split}
\iint_{\Omega\times\Omega}\frac{|u(x)-u(y)|^{\widetilde{p}}}{|x-y|^{N+\widetilde{s}\widetilde{p}}}\,dx\,dy \le\;& \left(\;\iint_{\Omega\times\Omega}\frac{|u(x)-u(y)|^p}{|x-y|^{N+sp}}\,dx\,dy \right)^\frac{\widetilde{p}}{p}
\left(\;\iint_{\Omega\times\Omega}\frac{dx\,dy }{|x-y|^{N-\frac{p\widetilde{p}(s-\widetilde{s})}{p-\widetilde{p}}}}\right)^\frac{p-\widetilde{p}}{p}\\
\le \;& C\left(\;\iint_{\Omega\times\Omega}\frac{|u(x)-u(y)|^p}{|x-y|^{N+sp}}\,dx\,dy \right)^\frac{\widetilde{p}}{p},
\end{split}
\end{equation*}
from which we obtain~\eqref{lemmaimmersione3}.

The claim of Lemma~\ref{immersionidisotto} in this case now follows from~\eqref{lemmaimmersione1} and~\eqref{lemmaimmersione3}.

We now focus on the case~$p=+\infty$.
If in addition~$\widetilde p=+\infty$, our claim becomes
\begin{equation*}
\|u\|_{C^{0,\widetilde{s}}(\Omega)}\leq C\|u\|_{C^{0,s}(\Omega)},
\end{equation*}
which in turn boils down to
\begin{equation}\label{6r7ewvhcdjxvgbcnhjf90loki8ju7hy6gtrfmjnhbtgvrfcd}
[u]_{C^{0,\widetilde{s}}(\Omega)}\leq C[u]_{C^{0,s}(\Omega)}.
\end{equation}
For this, we observe that, for any~$\widetilde s\in(0,s)$,
\begin{eqnarray*}
\frac{|u(x)-u(y)|}{|x-y|^{\widetilde s}}=
\frac{|u(x)-u(y)|}{|x-y|^{s}} |x-y|^{s-\widetilde s}\le [u]_{C^{0,s}(\Omega)} |x-y|^{s-\widetilde s}.
\end{eqnarray*}
Since~$\Omega$ is bounded, this implies~\eqref{6r7ewvhcdjxvgbcnhjf90loki8ju7hy6gtrfmjnhbtgvrfcd},
as desired.

If instead~$\widetilde p\in[1,p)=[1,+\infty)$, our claim becomes
\begin{equation}\label{mnbvcxasdfghjk09876543qwertyuio}
\|u\|_{W^{\widetilde{s},\widetilde{p}}(\Omega)}
\leq C\|u\|_{C^{0,s}(\Omega)}.
\end{equation}
In this situation, we point out that
$$ \|u\|_{L^{\widetilde p}(\Omega)}^p=\int_\Omega |u(x)|^{\widetilde p}\,dx\le \|u\|^{\widetilde p}_{L^\infty(\Omega)} |\Omega|.$$
Also,
\begin{eqnarray*}
\iint_{\Omega\times\Omega}\frac{|u(x)-u(y)|^{\widetilde p}}{|x-y|^{N+\widetilde s\widetilde p}}\,dx\,dy\le [u]^{\widetilde p}_{C^{0,s}(\Omega)}
\iint_{\Omega\times\Omega}\frac{|x-y|^{s\widetilde p}}{|x-y|^{N+\widetilde s\widetilde p}}\,dx\,dy\le C[u]^{\widetilde p}_{C^{0,s}(\Omega)}.
\end{eqnarray*}
{F}rom the last two displays, we obtain~\eqref{mnbvcxasdfghjk09876543qwertyuio}.
\end{proof}

\section{Proof of the embedding results}\label{sectionproofs}

In this section we prove our main embedding results.

\subsection{The case $sp<N$}
We first provide the following useful observation:

\begin{lemma}\label{896364389bcnxmnzfejEERRTC00}
Let~$s\in[0,1]$ and~$p\in[1,+\infty)$ be such that~$sp<N$. Let
\begin{equation}\label{mnbv-h5jk7j6u5h6}
\widetilde s\in[0,s]\qquad{\mbox{and}}\qquad\widetilde p\in\left[1,\frac{Np}{N-(s-\widetilde s)p}\right].\end{equation}
 
Let~$\gamma$ be the curve defined in~\eqref{curva1}.
Let~$0\leq \theta_1 < \theta_2 \leq 1$ and let~$\gamma(\theta_1)=(s_{\theta_1},p_{\theta_1})$
and~$\gamma(\theta_2)=(s_{\theta_2},p_{\theta_2})$. 

Then, 
\begin{equation}\label{6473829gfhdj098765bnvmc}
0\le s_{\theta_2}\le s_{\theta_1}\le 1,\qquad
\min\{p,\widetilde p\}\le p_{\theta_1},  p_{\theta_2}
\qquad{\mbox{and}}\qquad  p_{\theta_2}\le
\frac{Np_{\theta_1}}{N-(s_{\theta_1}-s_{\theta_2})p_{\theta_1}}.
\end{equation}

If~$p\le \widetilde p$, then~$p_{\theta_1}\le p_{\theta_2}$.

Furthermore, if in addition
\begin{equation}\label{bvgdyr658emdot0897ufhdn}
\widetilde p\ge \frac{sp}{sp-(p-1)\widetilde s},\end{equation}
then
\begin{equation}\label{bvgdyr658emdot0897ufhdn2}
p_{\theta_2}\ge \frac{s_{\theta_1} p_{\theta_1}}{s_{\theta_1} p_{\theta_1}-(p_{\theta_1}-1)s_{\theta_2}}.
\end{equation}

Moreover, $s_{\theta_1}=s_{\theta_2}$ if and only if~$s=\widetilde s$,
and~$p_{\theta_1}= p_{\theta_2}$ if and only if~$p=\widetilde p$.

Also,
$$ p_{\theta_2}=
\frac{Np_{\theta_1}}{N-(s_{\theta_1}-s_{\theta_2})p_{\theta_1}} $$
if and only if
$$\widetilde p= \frac{Np}{N-(s-\widetilde s)p}.$$
\end{lemma}

\begin{proof}
By inspection, one sees that~$0\le s_{\theta_2}\le s_{\theta_1}\le 1$.
Also, for all~$\theta\in[0,1]$,
\begin{equation}\label{4935634ewfgkjewbgbvsdnmvdmn}
\frac1{p_\theta}=
\frac{\widetilde{p}+\theta(p-\widetilde{p})}{p\widetilde{p}}
=\frac1p +\theta\left(\frac1{\widetilde p}-\frac1p\right)\le
\max\left\{\frac1p,\frac1{\widetilde p}\right\}=\frac1{\min\{p,\widetilde p\}}.
\end{equation}

Hence, to complete the proof of~\eqref{6473829gfhdj098765bnvmc},
it remains to check that
\begin{equation}\label{896364389bcnxmnzfejEERRTC}
p_{\theta_2}\le
\frac{Np_{\theta_1}}{N-(s_{\theta_1}-s_{\theta_2})p_{\theta_1}}.
\end{equation}
To this end, we first observe that, thanks to~\eqref{mnbv-h5jk7j6u5h6},
$$ \frac{N}{p}-\frac{N}{\widetilde p}\le s-\widetilde s.$$
This implies that
\begin{eqnarray*}&& \frac1{p_{\theta_1}}-\frac1{p_{\theta_2}}
=\frac{\widetilde p-\theta_1(\widetilde p-p)}{p\widetilde p}
-\frac{\widetilde p-\theta_2(\widetilde p-p)}{p\widetilde p}
=(\theta_2-\theta_1)\left(\frac{1}{p}-\frac{1}{\widetilde p}\right)
\\&&\qquad\qquad
\le \frac{(\theta_2-\theta_1)(s-\widetilde s)}{N}=\frac{s_{\theta_1}-s_{\theta_2}}{N}
.\end{eqnarray*}
{F}rom this, the desired result in~\eqref{896364389bcnxmnzfejEERRTC}
plainly follows.

If in addition~$p\le\widetilde p$,
by the definition of~$p_\theta$ we have that~$ p_{\theta_1}\le p_{\theta_2}$.

We now also assume~\eqref{bvgdyr658emdot0897ufhdn}
and we prove~\eqref{bvgdyr658emdot0897ufhdn2}.
For this, we introduce the notation
$$ \alpha:=1-\frac1p,\qquad\widetilde\alpha:=1-\frac1{\widetilde p}\qquad
{\mbox{and}}\qquad \beta:=s-\widetilde s$$
and we observe that
\begin{eqnarray*}&&
\left(1-\frac1{p_{\theta_1}}\right)s_{\theta_2}-
\left(1-\frac1{p_{\theta_2}}\right)s_{\theta_1}\\&=&
\big(\alpha-\theta_1(\alpha-\widetilde\alpha)\big)(s-\theta_2\beta)-
\big(\alpha-\theta_2(\alpha-\widetilde\alpha)\big)(s-\theta_1\beta)\\&=&
\beta\alpha(\theta_1-\theta_2)
+(\alpha-\widetilde\alpha)\big(\theta_2(s-\theta_1\beta)
-\theta_1(s-\theta_2\beta)\big)\\&=&
\beta\alpha(\theta_1-\theta_2)
+(\alpha-\widetilde\alpha)(\theta_2-\theta_1)s\\
&=&(\theta_2-\theta_1)\big(
\alpha( s-\beta)-\widetilde\alpha s\big)\\&=&
(\theta_2-\theta_1)\big(
\alpha\widetilde s-\widetilde\alpha s\big)\\&=&
(\theta_2-\theta_1)\left(
\left(1-\frac1p\right)\widetilde s-\left(1-\frac1{\widetilde p}\right)s\right).
\end{eqnarray*}
In light of this, we have that~\eqref{bvgdyr658emdot0897ufhdn}
implies~\eqref{bvgdyr658emdot0897ufhdn2}.

Finally, the last statements follow from the definition in~\eqref{curva1}.
\end{proof}

We now address the case in which the domain is~$\R^N$, by proving Theorem~\ref{teorema1}.

\begin{proof}[Proof of Theorem~\ref{teorema1}]
Let~$\widetilde{s}$ and~$\widetilde{p}$ satisfy~\eqref{insiemeimmersioneRN}. Then, by Lemma~\ref{curvaq11} there exist~$q\in[p,p^*_s]$ and~$\widetilde{\theta}\in [0,1]$ such that
\[
\widetilde{s}=(1-\widetilde{\theta})s \qquad \mbox{and} \qquad{\widetilde{p}}=
\frac{pq}{q-\widetilde\theta(q-p)}=\frac{pq}{p\widetilde\theta +q(1-\widetilde\theta)}.
\]
Hence, we can exploit Theorem~\ref{THMbrezismironuescu}
with~$s_1:=0$, $s_2:=s$, $p_1:=q$ and~$p_2:=p$ and obtain that
there exists a positive constant~$C_1=C_1(N, s, p, \widetilde s, \widetilde p)$ such that, for any~$u\in L^q(\R^N)\cap W^{s,p}(\R^N)$,
\begin{equation}\label{interpolazione11}
\|u\|_{W^{\widetilde{s},\widetilde{p}}(\R^N)}\le C_1 \|u\|_{L^q(\R^N)}^{\widetilde{\theta}}\|u\|_{W^{s,p}(\R^N)}^{1-\widetilde{\theta}}.
\end{equation}

Moreover, since~$q\in [p,p^*_s]$, the space~$W^{s,p}(\R^N)$ is continuously embedded in~$L^q(\R^N)$, namely
\[
\|u\|_{L^q(\R^N)}\le C_2\|u\|_{W^{s,p}(\R^N)}
\]
for some~$C_2=C_2(N, p, s)>0$ (see e.g.~\cite[Theorem 6.5]{MR2944369}). 

In light of this and~\eqref{interpolazione11}, for any~$u\in L^q(\R^N)\cap W^{s,p}(\R^N)=W^{s,p}(\R^N)$,
\[
\|u\|_{W^{\widetilde{s},\widetilde{p}}(\R^N)}\le C_1\, C_2^{\widetilde{\theta}}\|u\|_{W^{s,p}(\R^N)},
\]
which is~\eqref{inequality1}.

Having established the first claim of Theorem~\ref{teorema1},
we now focus on the second one.
For this, let~$0\leq \theta_1 < \theta_2 \leq 1$ and consider~$\gamma(\theta_1)=(s_{\theta_1},p_{\theta_1})$
and~$\gamma(\theta_2)=(s_{\theta_2},p_{\theta_2})$,
where the curve~$\gamma$ has been defined in~\eqref{curva1}. 
Thanks to the assumptions in~\eqref{insiemeimmersioneRN} and Lemma~\ref{896364389bcnxmnzfejEERRTC00}, we are in the position of 
exploiting~\eqref{inequality1}
with~$s:=s_{\theta_1}$, $p:=p_{\theta_1}$, $\widetilde s:=s_{\theta_2}$
and~$\widetilde p:=p_{\theta_2}$ and we obtain that~$
W^{s_{\theta_1}, p_{\theta_1}}(\R^N)$
is continuously embedded in~$W^{s_{\theta_2}, p_{\theta_2}}(\R^N)$,
as desired. This concludes the proof of Theorem~\ref{teorema1}.
\end{proof}

We now prove the continuous embedding in the case in which~$\Omega$ is an open bounded Lipschitz domain and~$sp<N$.

\begin{proof}[Proof of Theorem~\ref{teorema2}]
Let~$\widetilde s$ and~$\widetilde p$ satisfy~\eqref{insiemeimmersioneOmega}. Two cases can occur: either~$\widetilde p\ge p$ or~$\widetilde p< p$.

If~$\widetilde p\ge p$, then~$\widetilde s$ and~$\widetilde p$ satisfy~\eqref{insiemeimmersioneRN}. Hence, by Lemma~\ref{curvaq11} 
there exist~$q\in[p,p^*_s]$ and~$\widetilde{\theta}\in [0,1]$
such that
\[
\widetilde{s}=(1-\widetilde{\theta})s \qquad \mbox{and} \qquad{\widetilde{p}}=
\frac{pq}{q-\widetilde\theta(q-p)}=\frac{pq}{p\widetilde\theta +q(1-\widetilde\theta)}.
\]
{F}rom this and Theorem~\ref{THMbrezismironuescu} there exists a positive constant~$C_1=C_1(N, s, p, \widetilde s, \widetilde p)$ such that, for any~$u\in L^q(\Omega)\cap W^{s,p}(\Omega)$,
\begin{equation}\label{interpolazione12}
\|u\|_{W^{\widetilde{s},\widetilde{p}}(\Omega)}
\leq C_1 \|u\|_{L^q(\Omega)}^{\widetilde{\theta}}
\|u\|_{W^{s,p}(\Omega)}^{1-\widetilde{\theta}}.
\end{equation}

Moreover, since~$q\in [p,p^*_s]$, the space~$W^{s,p}(\Omega)$ is 
continuously embedded in~$L^q(\Omega)$, namely there exists~$C_2=C_2(N, s, p, \Omega)>0$ such that
\[
\|u\|_{L^q(\Omega)}\leq C_2\|u\|_{W^{s,p}(\Omega)},
\]
see e.g.~\cite[Theorem 6.7]{MR2944369}. Thus, using this into~\eqref{interpolazione12}, we conclude that,
for any~$u\in L^q(\Omega)\cap W^{s,p}(\Omega)=W^{s,p}(\Omega)$, 
\[
\|u\|_{W^{\widetilde{s},\widetilde{p}}(\Omega)}\le C_1\, C_2^{\widetilde{\theta}}\|u\|_{W^{s,p}(\Omega)}.
\]
which is~\eqref{inequality2}.

Let us consider the case~$\widetilde p< p$. If~$\widetilde s\in(0,s)$, then the desired embedding is a direct consequence of Lemma~\ref{immersionidisotto}. Otherwise, if~$\widetilde s=0$, then the desired result follows e.g. from~\cite[Theorem~6.7]{MR2944369}.

This completes the proof of the first part of Theorem~\ref{teorema2} and
we now focus on the second part.
For this, let~$0\leq \theta_1 < \theta_2 \leq 1$
and consider~$\gamma(\theta_1)=(s_{\theta_1},p_{\theta_1})$
and~$\gamma(\theta_2)=(s_{\theta_2},p_{\theta_2})$,
where the curve~$\gamma$ has been defined in~\eqref{curva1}. 
Thanks to the assumptions in~\eqref{insiemeimmersioneOmega}
and Lemma~\ref{896364389bcnxmnzfejEERRTC00},
we can use~\eqref{inequality2}
with~$s:=s_{\theta_1}$, $p:=p_{\theta_1}$, $\widetilde s:=s_{\theta_2}$
and~$\widetilde p:=p_{\theta_2}$ and
we obtain that~$
W^{s_{\theta_1}, p_{\theta_1}}(\Omega)$
is continuously embedded in~$W^{s_{\theta_2}, p_{\theta_2}}(\Omega)$,
as desired. This concludes the proof of Theorem~\ref{teorema2}.
\end{proof}

We now prove the compact embeddings stated in Theorem~\ref{teorema3} and Corollary~\ref{corollariocompattezza}.

\begin{proof}[Proof of Theorem~\ref{teorema3}]
Let~$\widetilde{s}$ and~$\widetilde{p}$ satisfy~\eqref{insiemeimmersionecompatta}. Then, from Lemma~\ref{curvaq13} there exist~$q\in[1,p^*_s)$ and~$\widetilde{\theta}\in (0,1]$ such that
\[
\widetilde{s}=(1-\widetilde{\theta})s \qquad \mbox{and} \qquad{\widetilde{p}}=
\frac{pq}{q-\widetilde\theta(q-p)}=\frac{pq}{p\widetilde\theta +q(1-\widetilde\theta)}.
\]
Hence, from Theorem~\ref{THMbrezismironuescu} there exists~$C>0$ such that, for any~$u\in L^q(\Omega)\cap W^{s,p}(\Omega)$,
\begin{equation}\label{interpolazione13}
\|u\|_{W^{\widetilde{s},\widetilde{p}}(\Omega)}\le C \|u\|_{L^q(\Omega)}^{\widetilde{\theta}}\|u\|_{W^{s,p}(\Omega)}^{1-\widetilde{\theta}}.
\end{equation}

Moreover, since~$q\in [1,p^*_s)$, the space~$W^{s,p}(\Omega)$ is compactly embedded in~$L^q(\Omega)$, namely for any sequence~$u_n$ in~$W^{s,p}(\Omega)$ which converges weakly to some~$u\in W^{s,p}(\Omega)$,  we have that~$u_n$ converges strongly to~$u$ in~$L^q(\Omega)$ (see e.g.~\cite[Theorem 7.2]{MR2944369}). 

Thus, by~\eqref{interpolazione13}, we infer that
\[
\|u_n -u \|_{W^{\widetilde{s},\widetilde{p}}(\Omega)}\le C \|u_n -u\|^{\widetilde{\theta}}_{L^q(\Omega)} \|u_n -u\|^{1-\widetilde{\theta}}_{W^{s,p}(\Omega)}.
\]
This implies that
\[
\lim_{n\to +\infty} \|u_n -u \|_{W^{\widetilde{s},\widetilde{p}}(\Omega)} =0,
\]
which proves the compact embedding of~$W^{s,p}(\Omega)$ in~$W^{\widetilde{s},\widetilde{p}}(\Omega)$.

Let now~$0<\theta_1<\theta_2\le1$ and consider~$\gamma(\theta_1)=(s_{\theta_1},p_{\theta_1})$
and~$\gamma(\theta_2)=(s_{\theta_2},p_{\theta_2})$,
where the curve~$\gamma$ has been defined in~\eqref{curva1}.
Thanks to~\eqref{insiemeimmersionecompatta}, we are in the position
of using Lemma~\eqref{896364389bcnxmnzfejEERRTC00}.
We thereby deduce that~$
W^{s_{\theta_1}, p_{\theta_1}}(\Omega) $
is compactly embedded in~$W^{s_{\theta_2}, p_{\theta_2}}(\Omega)$,
as desired. This concludes the proof of Theorem~\ref{teorema3}.
\end{proof}

\begin{proof}[Proof of Corollary~\ref{corollariocompattezza}]
Let~$\widetilde{s}$ and~$\widetilde{p}$ satisfy~\eqref{insiemeimmersionecompatta2}.
If they also satisfy~\eqref{insiemeimmersionecompatta}, then the desired result follows
from Theorem~\ref{teorema3}.

We can now assume that~$\widetilde{s}$ and~$\widetilde{p}$ 
satisfy~\eqref{insiemeimmersionecompatta2} but do not satisfy~\eqref{insiemeimmersionecompatta}, namely
\[
0\leq \widetilde{s}<s 
\qquad {\mbox{and}} \qquad
1 \le\widetilde{p}	< \frac{sp}{sp-(p-1)\widetilde{s}}.
\]
In this situation, we have that
$$
0\leq\frac{s+\widetilde{s}}{2}<s $$
and
$$ 1 \le\widetilde{p}
< \frac{sp}{sp-(p-1)\widetilde{s}}\le \frac{Np}{N-\frac{s-\widetilde s}2 p}.
$$
That is, $\frac{s+\widetilde{s}}{2}$ and~$\widetilde{p}$ satisfy~\eqref{insiemeimmersioneOmega},
and therefore we deduce from Theorem~\ref{teorema2}
that the space~$W^{s,p}(\Omega)$ is continuously embedded in~$W^{\frac{s+\widetilde{s}}{2},\widetilde{p}}(\Omega)$.

Thus, if~$u_n$ is a sequence in~$W^{s,p}(\Omega)$ which converges 
weakly to some~$u$ in~$W^{s,p}(\Omega)$, we have that~$u_n$ converges weakly to~$u$ in~$W^{\frac{s+\widetilde{s}}{2},\widetilde{p}}(\Omega)$.

Since~$\widetilde{s}<s$, we are also in the position of using Theorem~\ref{teorema3} and we infer that~$W^{\frac{s+\widetilde{s}}{2},\widetilde{p}}(\Omega)$ is 
compactly embedded in~$W^{\widetilde{s},\widetilde{p}}(\Omega)$.
Then, we conclude that~$u_n$ converges strongly to~$u$ in~$W^{\widetilde{s},\widetilde{p}}(\Omega)$, which completes the proof
of the first part of Corollary~\ref{corollariocompattezza}.

The second part of Corollary~\ref{corollariocompattezza}
is a consequence of the assumptions in~\eqref{insiemeimmersionecompatta2} and Lemma~\ref{896364389bcnxmnzfejEERRTC00}.
\end{proof}

\subsection{The case $sp=N$}

We now address the case~$sp=N$. In this case, we will use the following
observation:

\begin{lemma}\label{896364389bcnxmnzfejEERRTC00BIS}
Let~$s\in[0,1]$ and~$p\in[1,+\infty)$ be such that~$sp=N$. Let
\begin{equation*}
\widetilde s\in[0,s]\qquad{\mbox{and}}\qquad\widetilde p\in\left[1,\frac{N}{\widetilde s}\right].\end{equation*}
Let~$\gamma$ be the curve defined in~\eqref{curva1}.
Let~$0\leq \theta_1 < \theta_2 \leq 1$ and let~$\gamma(\theta_1)=(s_{\theta_1},p_{\theta_1})$
and~$\gamma(\theta_2)=(s_{\theta_2},p_{\theta_2})$. 

Then, 
\begin{equation*}
0\le s_{\theta_2}\le s_{\theta_1}\le 1,\qquad
\min\{p,\widetilde p\}\le p_{\theta_1}, p_{\theta_2}
\qquad {\mbox{and}}\qquad p_{\theta_2}\le
\frac{N}{s_{\theta_2}}.
\end{equation*}

Also, if~$p\le\widetilde p$, then~$p_{\theta_1}\le p_{\theta_2}$.

Moreover, $s_{\theta_1}=s_{\theta_2}$ if and only if~$s=\widetilde s$,
and~$p_{\theta_1}= p_{\theta_2}$ if and only if~$p=\widetilde p$.
\end{lemma}

\begin{proof}
By inspection, one sees that~$0\le s_{\theta_2}\le s_{\theta_1}\le 1$. Also, the fact that~$ p_{\theta}\ge \min\{p,\widetilde p\}$ for all~$\theta\in[0,1]$ follows from the
computations in~\eqref{4935634ewfgkjewbgbvsdnmvdmn}.

We now check that
\begin{equation}\label{754839uiohfdjskdns1234567890987opiuytr}
p_{\theta_2}\le\frac{N}{s_{\theta_2}}.
\end{equation}
For this, we observe that
$$ sp=N\ge \widetilde s\widetilde p$$
and therefore
\begin{eqnarray*}&&
\frac1{p_{\theta_2}}=\frac1p+\theta_2\left(\frac1{\widetilde p}-\frac1p\right)=\frac{1-\theta_2}p +\frac{\theta_2}{\widetilde p}
\ge \frac{(1-\theta_2)s}N +\frac{\theta_2\widetilde s}{N}=\frac{s_{\theta_2}}N,
\end{eqnarray*}
which gives~\eqref{754839uiohfdjskdns1234567890987opiuytr}.

Moreover, if~$p\le\widetilde p$, then, by the definition of~$p_\theta$ it follows that~$ p_{\theta_1}\le p_{\theta_2}$.
\end{proof}

We now deal with the case~$\Omega=\R^N$.

\begin{proof}[Proof of Theorem~\ref{teorema1sp=N}]
Let~$\widetilde{s}$ and~$\widetilde{p}$ satisfy~\eqref{insiemeimmersioneRNsp=N}. 
If~$\widetilde s \widetilde p =N$, then the result follows from~\cite[formula~(1.6) in Theorem~B]{MR3990737}. If not,
by Lemma~\ref{curvaq1sp=N} there exist~$q\in[p,+\infty)$ and~$\widetilde{\theta}\in (0,1]$ such that
\[
\widetilde{s}=(1-\widetilde{\theta})s \qquad \mbox{and} \qquad{\widetilde{p}}=
\frac{pq}{q-\widetilde\theta(q-p)}=\frac{pq}{p\widetilde\theta +q(1-\widetilde\theta)}.
\]
Thus, by Theorem~\ref{THMbrezismironuescu}, there exists a positive constant~$C_1=C_1(N, s, p, \widetilde s, \widetilde p)$ such that, for any~$u\in L^q(\R^N)\cap W^{s,p}(\R^N)$,
\begin{equation}\label{interpolazione11sp=N}
\|u\|_{W^{\widetilde{s},\widetilde{p}}(\R^N)}\le C_1 \|u\|_{L^q(\R^N)}^{\widetilde{\theta}}\|u\|_{W^{s,p}(\R^N)}^{1-\widetilde{\theta}}.
\end{equation}
Moreover, since~$q\in [p,+\infty)$, the space~$W^{s,p}(\R^N)$ is 
continuously embedded in~$L^q(\R^N)$, namely
\[
\|u\|_{L^q(\R^N)}\leq C_2\|u\|_{W^{s,p}(\R^N)}
\]
for some~$C_2=C_2(N, p, s)>0$ (see e.g.~\cite[Theorem~6.9]{MR2944369}). Therefore, for any~$u\in L^q(\R^N)\cap W^{s,p}(\R^N)=W^{s,p}(\R^N)$, from this and~\eqref{interpolazione11sp=N} we deduce that
\[
\|u\|_{W^{\widetilde{s},\widetilde{p}}(\R^N)}\leq C_1\, C_2^{\widetilde{\theta}}\|u\|_{W^{s,p}(\R^N)},
\]
which gives~\eqref{inequality1sp=N}.

Let now~$0\le \theta_1<\theta_2\le1$ and consider~$\gamma(\theta_1)=(s_{\theta_1},p_{\theta_1})$
and~$\gamma(\theta_2)=(s_{\theta_2},p_{\theta_2})$,
where the curve~$\gamma$ has been defined in~\eqref{curva1}.
Then, Lemma~\ref{896364389bcnxmnzfejEERRTC00BIS}
gives that we can use~\eqref{inequality1sp=N}
with~$s:=s_{\theta_1}$, $p:=p_{\theta_1}$, $\widetilde s:=s_{\theta_2}$
and~$\widetilde p:=p_{\theta_2}$. We thereby obtain that~$
W^{s_{\theta_1}, p_{\theta_1}}(\R^N)$
is continuously embedded in~$W^{s_{\theta_2}, p_{\theta_2}}(\R^N)$,
as desired.

We stress that in the case $N=1$,
one has that the space~$W^{1, 1}(\R^N)$ is continuously embedded in~$L^\infty(\R^N)$ (see e.g.~\cite[Theorem~8.8]{MR2759829}).
This concludes the proof of Theorem~\ref{teorema1sp=N}.
\end{proof}

\begin{proof}[Proof of Corollary~\ref{corollaryBMO}]
In light of Theorem~\ref{teorema1sp=N}, we only need to show that,
for all~$u\in W^{\widetilde s, \widetilde p}(\R^N)$,
\begin{equation}\label{claimBMO}
\|u\|_{BMO}\le C_1 \|u\|_{W^{\widetilde s, \widetilde p}(\R^N)}.
\end{equation}
To this end, we observe that, for any~$r>0$ and for any~$x\in\R^N$,
\[
\begin{split}
\iint_{B_r (x)\times B_r (x)} |u(y)-u(z)| \, dy\, dz & = \iint_{B_r (x)\times B_r (x)}
\frac{|u(y)-u(z)|}{|y-z|^{\frac{2N}{\widetilde p}}} \, |y-z|^{\frac{2N}{\widetilde p}}\, dy \,dz\\
&\le (2r)^{\frac{2N}{\widetilde p}}
\iint_{B_r (x)\times B_r (x)}
\frac{|u(y)-u(z)|}{|y-z|^{\frac{2N}{\widetilde p}}} \, dy \,dz\\
&\le (2r)^{\frac{2N}{\widetilde p}}
|B_r|^{\frac{2\widetilde p -2}{\widetilde p}}
\left(\;\iint_{B_r (x)\times B_r (x)} 
\frac{|u(y)-u(z)|^{\widetilde p}}{|y-z|^{2N}} \, dy\, dz\right)^{\frac{1}{\widetilde p}} 
\\
&\le (2r)^{\frac{2N}{\widetilde p}} \, |B_r|^{2-\frac{2}{\widetilde p}} \, \|u\|_{W^{\widetilde s, \widetilde p}(\R^N)}.
\end{split}
\]
Dividing by~$|B_r|^2$ and
taking the supremum over~$r>0$ and~$x\in\R^N$, we get the desired result in~\eqref{claimBMO}.
\end{proof}

We now consider the case in which~$sp=N$ and~$\Omega$ is an open and bounded domain Lipschitz boundary.

\begin{proof}[Proof of Theorem~\ref{teorema2sp=N}]
Let~$\widetilde s$ and~$\widetilde p$ satisfy~\eqref{insiemeimmersioneOmegasp=N}. Then, either~$\widetilde p\ge p$ or~$\widetilde p< p$.

Suppose first that~$\widetilde p\ge p$.
If in addition~$\widetilde s \widetilde p =N$, then the desired
result follows from~\cite[formula~(1.6) in Theorem~B]{MR3990737}.

If instead~$\widetilde s \widetilde p <N$, we observe that~$\widetilde s$ and~$\widetilde p$ satisfy~\eqref{insiemeimmersioneRNsp=N},
and therefore we are in the position of
exploiting Lemma~\ref{curvaq1sp=N}. In this way, we obtain that 
there exist~$q\in[p,+\infty)$ and~$\widetilde{\theta}\in (0,1]$
such that
\[
\widetilde{s}=(1-\widetilde{\theta})s \qquad \mbox{and} \qquad{\widetilde{p}}=
\frac{pq}{q-\widetilde\theta(q-p)}=\frac{pq}{p\widetilde\theta +q(1-\widetilde\theta)}.
\]
{F}rom this and Theorem~\ref{THMbrezismironuescu} there exists a positive constant~$C_1=C_1(N, s, p, \widetilde s, \widetilde p)$ such that, for any~$u\in L^q(\Omega)\cap W^{s,p}(\Omega)$,
\begin{equation}\label{interpolazione12sp=N}
\|u\|_{W^{\widetilde{s},\widetilde{p}}(\Omega)}\le C \|u\|_{L^q(\Omega)}^{\widetilde{\theta}}\|u\|_{W^{s,p}(\Omega)}^{1-\widetilde{\theta}}.
\end{equation}
Moreover, since~$q\in [p,+\infty)$, the space~$W^{s,p}(\Omega)$ is 
continuously embedded in~$L^q(\Omega)$, namely there exists~$C_2=C_2(N, s, p, \Omega)>0$ such that
\[
\|u\|_{L^q(\Omega)}\le C_2\|u\|_{W^{s,p}(\Omega)}
,\]
see e.g.~\cite[Theorem 6.10]{MR2944369}. Thus, for any~$u\in L^q(\Omega)\cap W^{s,p}(\Omega)=W^{s,p}(\Omega)$, this and~\eqref{interpolazione12sp=N} give that
\[
\|u\|_{W^{\widetilde{s},\widetilde{p}}(\Omega)}\le C_1\, C_2^{\widetilde{\theta}}\|u\|_{W^{s,p}(\Omega)},
\]
which establishes~\eqref{inequality2sp=N}.

Let us consider the case~$\widetilde p< p$. If~$\widetilde s>0$, then the desired embedding is a direct consequence of Lemma~\ref{immersionidisotto}. Otherwise, if~$\widetilde s=0$, the desired result is the one stated e.g. in~\cite[Theorem~6.10]{MR2944369}.

Let now~$0\le \theta_1<\theta_2\le1$ and consider~$\gamma(\theta_1)=(s_{\theta_1},p_{\theta_1})$
and~$\gamma(\theta_2)=(s_{\theta_2},p_{\theta_2})$,
where the curve~$\gamma$ has been defined in~\eqref{curva1}.
Then, Lemma~\ref{896364389bcnxmnzfejEERRTC00BIS}
gives that we can use~\eqref{inequality2sp=N}
with~$s:=s_{\theta_1}$, $p:=p_{\theta_1}$, $\widetilde s:=s_{\theta_2}$
and~$\widetilde p:=p_{\theta_2}$. We thereby obtain that~$
W^{s_{\theta_1}, p_{\theta_1}}(\Omega)$
is continuously embedded in~$W^{s_{\theta_2}, p_{\theta_2}}(\Omega)$,
as desired.

We stress that in the case~$N=1$,
the space~$W^{1, 1}(\Omega)$ is continuously embedded in~$L^\infty(\Omega)$ (see e.g.~\cite[Theorem~8.8]{MR2759829}).
This concludes the proof of Theorem~\ref{teorema2sp=N}.
\end{proof}

Now, we prove the desired compact embedding in the case~$sp=N$.

\begin{proof}[Proof of Theorem~\ref{teorema3sp=N}]
Let~$\widetilde{s}$ and~$\widetilde{p}$ satisfy~\eqref{2insiemeimmersioneOmegasp=N}.

If~$\widetilde{p}=p$, the desired compact embedding is a direct 
consequence of Lemma~\ref{orizzontalecompatto}. Thus, we now
distinguish
the cases~$\widetilde{p}>p$ and~$\widetilde{p}<p$.

If~$\widetilde{p}>p$, we set~$\overline{s}:={N}/{\widetilde{p}}$
and we observe that~$\overline{s}\in(0,s)$.
We can therefore apply Theorem~\ref{teorema2sp=N} and conclude that~$
W^{s,p}(\Omega)$ is continuously embedded in~$W^{\overline{s},\widetilde{p}}(\Omega)$.
Moreover,
we have that~$\widetilde s\in[0,\overline s)$, 
and so we deduce from Lemma~\ref{orizzontalecompatto} that~$
W^{\overline{s},\widetilde{p}}(\Omega)$
is compactly embedded in~$W^{\widetilde{s},\widetilde{p}}(\Omega)$.
Combining these pieces of information, 
we obtain that the space~$W^{s,p}(\Omega)$ is compactly embedded in~$W^{\widetilde{s},\widetilde{p}}(\Omega)$, as desired.

Now, we deal with the case~$\widetilde{p}<p$.
We observe that~$\frac{s+\widetilde{s}}{2}\in(0,s)$
and that
$$ \widetilde p<p=\frac{N}{s}< \frac{2N}{s+\widetilde s}.$$
Thus, we can apply 
Theorem~\ref{teorema2sp=N} to see that~$W^{s,p}(\Omega)$ is continuously embedded in~$W^{\frac{s+\widetilde{s}}{2},\widetilde{p}}(\Omega)$.
Moreover, since~$\frac{s+\widetilde{s}}{2}>\widetilde{s}$, from Lemma~\ref{orizzontalecompatto}
we infer that~$W^{\frac{s+\widetilde{s}}{2},\widetilde{p}}(\Omega)$
is compactly embedded in~$W^{\widetilde{s},\widetilde{p}}(\Omega)$.
{F}rom these considerations, 
we obtain that the space~$W^{s,p}(\Omega)$ is compactly embedded in~$W^{\widetilde{s},\widetilde{p}}(\Omega)$, which concludes the case~$\widetilde{p}<p$.

The first statement of Theorem~\ref{teorema3sp=N} is thereby
proved, and we now focus on the second one.

Let~$0< \theta_1<\theta_2\le1$ and consider~$\gamma(\theta_1)=(s_{\theta_1},p_{\theta_1})$
and~$\gamma(\theta_2)=(s_{\theta_2},p_{\theta_2})$,
where the curve~$\gamma$ has been defined in~\eqref{curva1}.
Then, Lemma~\ref{896364389bcnxmnzfejEERRTC00BIS}
gives that we can use the first statement of Theorem~\ref{teorema3sp=N} 
with~$s:=s_{\theta_1}$, $p:=p_{\theta_1}$, $\widetilde s:=s_{\theta_2}$
and~$\widetilde p:=p_{\theta_2}$ and obtain that~$
W^{s_{\theta_1}, p_{\theta_1}}(\Omega)$
is compactly embedded in~$W^{s_{\theta_2}, p_{\theta_2}}(\Omega)$,
as desired.
\end{proof}

\subsection{The case $sp>N$}

We now consider the case~$sp>N$. We provide some auxiliary statements.

\begin{lemma}\label{lemmaspr1}
Let~$s\in(0,1]$ and~$p\in[1,+\infty]$ be such that~$sp>N$. 
Let
\begin{equation}\label{step1sptilde}
\overline s \in \left(\frac{sp-N}{p}, s\right]
\qquad \mbox{and}\qquad
\overline p := \frac{Np}{N-(s-\overline s)p}.
\end{equation}
Let~$\Omega$ be either~$\R^N$ or a bounded Lipschitz domain of~$\R^N$.

Then, $W^{s, p}(\Omega)$ is continuously embedded in~$W^{\overline s, \overline p}(\Omega)$.
\end{lemma}

\begin{proof}
To this end, notice that if~$\overline s=s$ then~$\overline p=p$,
and therefore the claim is trivial.

Otherwise, we observe that,
if~$\overline s$ and~$ \overline p$ satisfy~\eqref{step1sptilde},
then the point~$(\overline s, \overline p)$ belongs to the curve
\[
\left(s-\frac{\theta N}{p}, \frac{p}{1-\theta}\right)\quad {\mbox{with }} \theta\in[0,1].
\]
This says that we are in the setting of
Theorem~\ref{THMbrezismironuescu} with~$s_1:=\frac{sp-N}{p}$,
$p_1:=+\infty$, $s_2:=s$ and~$p_2:=p$. Therefore,
we can exploit the claim of Theorem~\ref{THMbrezismironuescu}
with~$s_\theta:=\overline s$ and~$p_\theta:=\overline p$.
In this way, we obtain that 
there exists a constant~$C=C(N,s,p,\theta)>0$ such that
\begin{equation*}
\|u\|_{W^{\overline s, \overline p}(\Omega)}
\leq C \|u\|_{C^{0,\frac{sp-N}{p}}(\Omega)}^\theta
\|u\|_{W^{ s,  p}(\Omega)}^{1-\theta}.
\end{equation*}
{F}rom this and the continuous embedding of~$W^{ s,  p}(\Omega)$ in~$C^{0,\frac{sp-N}{p}}(\Omega)$ (see~\cite[Theorem~8.2]{MR2944369}),
we thus conclude that
\[
\|u\|_{W^{\overline s, \overline p}(\Omega)}
\leq C \|u\|_{W^{ s,  p}(\Omega)},
\]
which gives the desired result.
\end{proof}

\begin{lemma}\label{lemmaspr2}
Let~$s\in(0,1]$ and~$p\in[1,+\infty]$.
Let~$\widetilde s\in[0,s)$.
Let~$\Omega$ be either~$\R^N$ or a bounded Lipschitz domain of~$\R^N$.

Then, $W^{s, p}(\Omega)$ is continuously embedded in~$W^{\widetilde s,  p}(\Omega)$.
\end{lemma}

\begin{proof}
We notice that the point~$(\widetilde s, p)$ belongs to the curve~$
\big((1-\theta)s, p\big)$, with~$\theta\in(0,1]$.
As a consequence,
we can use Theorem~\ref{THMbrezismironuescu} with~$s_1:=0$,
$p_1:=p$, $s_2:=s$, $p_2:=p$, $s_\theta:=\widetilde s$ and~$p_\theta:=p$, obtaining that there exists a constant~$C=C(N,s,p,\theta)>0$ such that
\begin{equation*}
\|u\|_{W^{\widetilde s, p}(\Omega)}
\leq C  \|u\|_{L^p(\Omega)}^\theta
\|u\|_{W^{ s,  p}(\Omega)}^{1-\theta}\le C\|u\|_{W^{ s,  p}(\Omega)},
\end{equation*}
which implies the desired claim.
\end{proof}

We will also need the following observation:

\begin{lemma}\label{896364389bcnxmnzfejEERRTC00TER}
Let~$s\in(0,1]$ and~$p\in[1,+\infty]$ be such that~$sp>N$. Let~$\widetilde s\in[0,s]$ and~$\widetilde p\in[1,+\infty]$.

Let~$\gamma$ be the curve defined in~\eqref{curva1}.
Let~$0\leq \theta_1 < \theta_2 \leq 1$ and let~$\gamma(\theta_1)=(s_{\theta_1},p_{\theta_1})$
and~$\gamma(\theta_2)=(s_{\theta_2},p_{\theta_2})$. 

Then, we have that~$0\le s_{\theta_2}\le s_{\theta_1}\le 1$ and that~$
\min\{p,\widetilde p\}\le p_{\theta_1}, p_{\theta_2}\le+\infty$.

If~$p\le\widetilde p$, then~$p_{\theta_1}\le p_{\theta_2}$. 

Furthermore, if
\begin{equation}\label{432i1ohejwfbdnsDFRSTRWTI954078}
\dfrac{sp-N}{p}< \widetilde s\le s \qquad{\mbox{and}}\qquad
1\le \widetilde p\le \dfrac{Np}{N-(s-\widetilde s)p},
\end{equation}
then
\begin{equation}\label{432i1ohejwfbdnsDFRSTRWTI9540782}
\dfrac{s_{\theta_1}p_{\theta_1}-N}{p_{\theta_1}}< s_{\theta_2}\le s_{\theta_1} \qquad{\mbox{and}}
\qquad
1\le p_{\theta_2}\le \dfrac{Np_{\theta_1}}{N-(s_{\theta_1}-s_{\theta_2})p_{\theta_1}}.
\end{equation}

Moreover, $s_{\theta_1}=s_{\theta_2}$ if and only if~$s=\widetilde s$,
and~$p_{\theta_1}= p_{\theta_2}$ if and only if~$p=\widetilde p$.

Also,
$$ p_{\theta_2}=
\frac{Np_{\theta_1}}{N-(s_{\theta_1}-s_{\theta_2})p_{\theta_1}} $$
if and only if
$$\widetilde p= \frac{Np}{N-(s-\widetilde s)p}.$$
\end{lemma}

\begin{proof}
By inspection, one sees that~$0\le s_{\theta_2}\le s_{\theta_1}\le 1$
and~$ p_{\theta_1}$, $p_{\theta_2}\le+\infty$. Also, the fact that~$ p_{\theta}\ge \min\{p,\widetilde p\}$ for all~$\theta\in[0,1]$ follows from the
computations in~\eqref{4935634ewfgkjewbgbvsdnmvdmn}.

If in addition~$p\le\widetilde p$, by the definition of~$p_{\theta}$
it follows that~$p_{\theta_1}\le p_{\theta_2}$.

We now check that~\eqref{432i1ohejwfbdnsDFRSTRWTI954078}
implies~\eqref{432i1ohejwfbdnsDFRSTRWTI9540782}.
For this, we point out that
\begin{eqnarray*}&&
s_{\theta_2}-
\dfrac{s_{\theta_1}p_{\theta_1}-N}{p_{\theta_1}}
=s_{\theta_2}-s_{\theta_1}+\frac{N}{p_{\theta_1}}
= (\theta_2-\theta_1)(\widetilde s-s) +\frac{N}{p}+N\theta_1\left(\frac1{\widetilde p}-\frac1p\right)\\&&\qquad\qquad
\ge (\theta_2-\theta_1)(\widetilde s-s) +\frac{N}{p}
-N\theta_1\left(\frac{s-\widetilde s}N\right)\\&&\qquad\qquad=
\theta_2(\widetilde s-s)+\frac{N}{p}
>-\frac{N\theta_2}p+\frac{N}{p}\ge0,
\end{eqnarray*}
which gives the first claim in~\eqref{432i1ohejwfbdnsDFRSTRWTI9540782}.

Also,
\begin{eqnarray*}
&&\frac1{p_{\theta_2}}- \dfrac{N-(s_{\theta_1}-s_{\theta_2})p_{\theta_1}}{Np_{\theta_1}} =
\frac1{p_{\theta_2}}-\frac1{p_{\theta_1}}  +\dfrac{s_{\theta_1}-s_{\theta_2}}{N} \\&&\quad=
(\theta_2-\theta_1)\left(\frac1{\widetilde p}-\frac1p\right)
+\frac{(\theta_2-\theta_1)(s-\widetilde s)}N
\ge (\theta_2-\theta_1)\left(\frac{\widetilde s-s}N
+\frac{s-\widetilde s}N
\right)=0
\end{eqnarray*}
which gives the second claim in~\eqref{432i1ohejwfbdnsDFRSTRWTI9540782}. 

Finally,  the last statement follows from the definition in~\eqref{curva1}.
\end{proof}

We now establish Theorems~\ref{teorema1sp>N} and~\ref{teorema2sp>N}.

\begin{proof}[Proof of Theorems~\ref{teorema1sp>N}
and~\ref{teorema2sp>N}]
In what follows, $\Omega$ is either~$\R^N$ or an open and bounded
domain with Lipschitz boundary. In the first case, we suppose that~\eqref{insiemeimmersioneRNsp>N} is in force, while in the latter case
we assume~\eqref{insiemeimmersioneOmegasp>N}.

We provide the proof of the first statement of Theorems~\ref{teorema1sp>N}
and~\ref{teorema2sp>N}.

We first consider the case~$\widetilde p=+\infty$. Then, we are in the case~$\widetilde s\in[0,(sp-N)/{p}]$.
{F}rom~\cite[Theorem~8.2]{MR2944369} we know that~$W^{s,  p}(\Omega)$ is continuously embedded in~$C^{0,\frac{sp-N}{p}}(\Omega)$.
Moreover, we have that~$C^{0,\frac{sp-N}{p}}(\Omega)$ is continuously embedded in~$C^{0,\widetilde s}(\Omega)$. These considerations prove
that~$W^{s,  p}(\Omega)$ is continuously embedded in~$C^{0,\widetilde s}(\Omega)$, which is the desired statement when~$\widetilde p=+\infty$.

Now, we deal with the case~$\widetilde p\in [p,+\infty)$.
We set
\begin{equation}\label{desbarra06849}
\overline s:=s-N\left(\frac{1}{p}-\frac{1}{\widetilde p}\right)\end{equation}
and we point out that
$$ \overline s\in\left(\frac{sp-N}{p},s\right]\qquad
{\mbox{and}}\qquad \widetilde p = \frac{Np}{N-(s-\overline s)p}.$$
Thus, we are in the position of using Lemma~\ref{lemmaspr1}
with~$\overline p$ replaced by~$\widetilde p$ and we deduce that~$W^{ s,  p}(\Omega)$
is continuously embedded in~$W^{s-N\left(\frac{1}{p}-\frac{1}{\widetilde p}\right) , \widetilde p}(\Omega)$.

Also, we claim that
\begin{equation}\label{poiuytre7843887654eewfgcvyu23}
\overline s\in[\widetilde s, s].
\end{equation}
Indeed, since~$p\le\widetilde p$, we have that~$\overline s\le s$.

Moreover, if~$\widetilde s\in[0,(sp-N)/p]$, we see that
\begin{equation}\label{bnvcweihtu54y6584} \widetilde s\leq s-\frac{N}p < s-N\left(\frac{1}{p}-\frac{1}{\widetilde p}\right)=\overline s.
\end{equation}
If instead~$\widetilde s\in((sp-N)/p,s]$, we use the condition on~$\widetilde p$ in~\eqref{insiemeimmersioneRNsp>N} to see that
$$ \frac1{\widetilde p}\ge \frac1p-\frac{s-\widetilde s}{p}
$$ and therefore
\begin{equation}\label{bnvcweihtu54y65842}
\widetilde s\le s= s-\frac{N}p+\frac{N}p\le
s-\frac{N}p+\frac{N}{\widetilde p}+\frac{(s-\widetilde s)N}{p}
=\overline s+\frac{(s-\widetilde s)N}{p}\le \overline s.
\end{equation}
Gathering these observations, we obtain~\eqref{poiuytre7843887654eewfgcvyu23}.

Thanks to~\eqref{poiuytre7843887654eewfgcvyu23}, we can exploit Lemma~\ref{lemmaspr2} with~$p$
replaced by~$\widetilde p$ and~$s$ replaced by~$\overline s$.
In this way, we obtain that~$W^{s-N\left(\frac{1}{p}-\frac{1}{\widetilde p}\right) , \widetilde p}(\Omega)$ is continuously embedded in~$W^{\widetilde s , \widetilde p}(\Omega)$.
As a result, we conclude in this case that~$W^{ s,  p}(\Omega)$
is continuously embedded in~$W^{\widetilde s, \widetilde p}(\Omega)$,
as desired.

If~$\Omega$ is an open and bounded domain with Lipschitz boundary,
we also consider the case~$\widetilde p\in[1,p)$. In this situation,
we employ Lemma~\ref{immersionidisotto} and we deduce that~$W^{ s,  p}(\Omega)$
is continuously embedded in~$W^{\widetilde s, \widetilde p}(\Omega)$.

This concludes the proof of the first statement of Theorems~\ref{teorema1sp>N}
and~\ref{teorema2sp>N}.

We now prove the second claim of Theorems~\ref{teorema1sp>N}
and~\ref{teorema2sp>N}.
For this, take~$0\leq \theta_1 < \theta_2 \leq 1$
and consider~$\gamma(\theta_1)=(s_{\theta_1},p_{\theta_1})$
and~$\gamma(\theta_2)=(s_{\theta_2},p_{\theta_2})$,
where the curve~$\gamma$ has been defined in~\eqref{curva1}.
Lemma~\ref{896364389bcnxmnzfejEERRTC00TER} allows us
to use the first statement of either Theorem~\ref{teorema1sp>N}
or Theorem~\ref{teorema2sp>N}
with~$s:=s_{\theta_1}$, $p:=p_{\theta_1}$, $\widetilde s:=s_{\theta_2}$
and~$\widetilde p:=p_{\theta_2}$ and obtain the desired embedding.
\end{proof}

Now, we prove the desired compact embedding in the case~$sp>N$.

\begin{proof}[Proof of Theorem~\ref{teorema3sp>N}]
Let~$\widetilde{s}$ and~$\widetilde{p}$ satisfy~\eqref{2insiemeimmersioneOmegasp>N}.

If~$\widetilde{p}=p$, the desired compact embedding is a direct 
consequence of Lemma~\ref{orizzontalecompatto}. Thus, we now distinguish
the cases~$\widetilde{p}>p$ and~$\widetilde{p}<p$.

If~$\widetilde{p}\in(p,+\infty)$, we define~$\overline{s}$ as in~\eqref{desbarra06849}
and we recall that~$\overline s\in[\widetilde s,s)$ (thanks to~\eqref{poiuytre7843887654eewfgcvyu23} and the fact that~$\widetilde p\neq p$).

Also, if in particular~$\overline s\in( (sp-N)/p,s)$, we have that
\begin{eqnarray*}
\frac{Np}{N-(s-\overline s)p} = \frac{Np}{N-N\left(\frac1p-\frac1{\widetilde p}\right)p}=\frac{p}{1-\left(\frac1p-\frac1{\widetilde p}\right)p}=\widetilde p.
\end{eqnarray*}
As a result, we find that~$\overline s$ and~$\widetilde p$ satisfy
the assumptions in~\eqref{insiemeimmersioneOmegasp>N}. Therefore,
we can apply Theorem~\ref{teorema2sp>N} and conclude that~$W^{s,p}(\Omega)$ is continuously embedded in~$W^{\overline{s},\widetilde{p}}(\Omega)$.

Notice also that~$\overline s>\widetilde s$
(recall the computations in~\eqref{bnvcweihtu54y6584} and~\eqref{bnvcweihtu54y65842} and use the fact that~$ \widetilde s<s$).
Hence, we deduce from Lemma~\ref{orizzontalecompatto} that~$
W^{\overline{s},\widetilde{p}}(\Omega) $
is compactly embedded in~$W^{\widetilde{s},\widetilde{p}}(\Omega)$.
Combining these pieces of information,
we infer that the space~$W^{s,p}(\Omega)$ is compactly embedded in~$W^{\widetilde{s},\widetilde{p}}(\Omega)$, as desired.

Now, we deal with the case~$\widetilde{p}<p$.
We observe that~$\frac{s+\widetilde{s}}{2}\in(0,s)$.
Moreover, if~$\frac{s+\widetilde{s}}{2}\in((sp-N)/p,s)$,
\begin{eqnarray*}&&
\frac1{\widetilde p}>\frac1p>\frac1p-\frac{s-\widetilde s}{2N}.
\end{eqnarray*}
Hence, $\frac{s+\widetilde{s}}{2}$ and~$\widetilde p$ satisfy
the assumptions in~\eqref{insiemeimmersioneOmegasp>N}. Therefore,
we can apply Theorem~\ref{teorema2sp>N}
to see that~$W^{s,p}(\Omega) $ is continuously embedded in~$
W^{\frac{s+\widetilde{s}}{2},\widetilde{p}}(\Omega)$.

Moreover, since~$\frac{s+\widetilde{s}}{2}>\widetilde{s}$, from Lemma~\ref{orizzontalecompatto} we infer that~$W^{\frac{s+\widetilde{s}}{2},\widetilde{p}}(\Omega)$
is compactly embedded in~$W^{\widetilde{s},\widetilde{p}}(\Omega)$.
Gathering these considerations, 
we conclude that the space~$W^{s,p}(\Omega)$ is compactly embedded in~$W^{\widetilde{s},\widetilde{p}}(\Omega)$, which concludes the case~$\widetilde{p}<p$.

Now, let~$0< \theta_1<\theta_2\le1$ and consider~$\gamma(\theta_1)=(s_{\theta_1},p_{\theta_1})$
and~$\gamma(\theta_2)=(s_{\theta_2},p_{\theta_2})$,
where the curve~$\gamma$ has been defined in~\eqref{curva1}.
Then, Lemma~\ref{896364389bcnxmnzfejEERRTC00TER}
gives that we can use the first statement of Theorem~\ref{teorema3sp>N} 
with~$s:=s_{\theta_1}$, $p:=p_{\theta_1}$, $\widetilde s:=s_{\theta_2}$
and~$\widetilde p:=p_{\theta_2}$ and obtain that~$
W^{s_{\theta_1}, p_{\theta_1}}(\Omega)$
is compactly embedded in~$W^{s_{\theta_2}, p_{\theta_2}}(\Omega)$,
as desired.
\end{proof}

\section*{Acknowledgements} 
All the authors are members of the Australian Mathematical Society (AustMS). CS and EPL are members of the INdAM--GNAMPA.

This work has been supported by the Australian Laureate Fellowship FL190100081 and by the Australian Future Fellowship
FT230100333.

\vfill

\end{document}